\documentclass[reqno]{amsart} 
\usepackage[utf8]{inputenc} 

\usepackage{amssymb,amsthm,amsfonts}
\usepackage{mathtools}
\usepackage{fixltx2e} 
\usepackage{tikz} 
\usepackage{graphicx} 
\usepackage{psfrag} 
\usepackage{caption}
\usepackage{subcaption}
\usepackage{cancel} 
\usepackage{hyperref} 
\usepackage{comment} 
\usepackage{accents} 
\usepackage{ulem}

\usepackage{calrsfs}

\newtheorem{theorem}{Theorem}[section] 
\newtheorem{lemma}{Lemma}[section] 
\newtheorem{proposition}{Proposition}[section] 
\newtheorem{corollary}{Corollary}[section] 
\newtheorem{remark}{Remark}[section] 
 
\newtheorem{eg}{Example}[section] 
\newtheorem{claim}{Claim}
\numberwithin{equation}{section}

\newlength{\dhatheight}

\newcommand{\caret}{\mathbin{\scriptscriptstyle\wedge}}

\newcommand{\kay}{\ensuremath{k}}

\renewcommand{\H}{\ensuremath{\mathsf{H}}}

\renewcommand{\L}{\ensuremath{\mathsf{L}}}
\newcommand{\M}{\ensuremath{\mathsf{M}}}
\newcommand{\N}{\ensuremath{\mathsf{N}}}

\newcommand{\E}{\ensuremath{\mathsf{E}}}
\newcommand{\F}{\ensuremath{\mathsf{F}}}
\newcommand{\G}{\ensuremath{\mathsf{G}}}
\newcommand{\I}{\ensuremath{\mathsf{I}}}
\newcommand{\J}{\ensuremath{\mathsf{J}}}
\newcommand{\K}{\ensuremath{\mathsf{K}}}
\newcommand{\U}{\ensuremath{\mathsf{U}}}
\newcommand{\V}{\ensuremath{\mathsf{V}}}
\newcommand{\A}{\ensuremath{\mathsf{A}}}

\newcommand{\C}{\ensuremath{\mathsf{C}}}

\renewcommand{\mod}{\ensuremath{\, \, \mathrm{mod} \,}}
\newcommand{\sgn}{\ensuremath{\epsilon}}
\newcommand{\discr}{\ensuremath{\mathrm{discr}}}
\newcommand{\card}{\ensuremath{\#}}
\newcommand{\rad}{\ensuremath{\mathrm{rad}}}

\title[]{Quadratic Irrationals, Generating Functions and L{\'e}vy constants}


\thanks{
This work has been partially supported by CAPES Special Visiting Researcher grant CSF-PVE-S - 88887.117899/2016-00.
}

\author{Anna Belova} 
\address{Anna Belova, Mathematics Department, Uppsala University, Uppsala, Sweden}
\email[]{anna.belova@math.uu.se}

\author{Peter Hazard} 
\address{Peter Hazard, Instituto de Matem\'{a}tica e Estat{\'i}stica, USP, S\~{a}o Paulo, SP, Brazil}
\email[]{pete@ime.usp.br}

\date{\today}

\begin{document}

\begin{abstract}
We show that the generating function corresponding 
to the sequence of denominators of the best rational 
approximants of a quadratic irrational
is a rational function with integer coefficients.
Consequently we can compute the L{\'e}vy constant 
of any quadratic irrational explicitly in terms of 
a finite number of its convergents.
\end{abstract}

\maketitle

\section{Introduction}
\subsection{Background.}
The aim of this article is to show the following 
theorem.
\begin{theorem}\label{thm:main-gen_fn_rat}
Let 
$\theta\in\mathbb{R}\setminus\mathbb{Q}$ 
be a quadratic irrational.
For each $n$, let $p_n/q_n$ denote the $n$th 
best rational approximant to $\theta$.
Then the generating functions
\begin{equation}
F(z)\;=\;\sum_{n\geq 0}p_n z^n \qquad \mbox{and} \qquad 
G(z)\;=\;\sum_{n\geq 0}q_n z^n
\end{equation}
are both rational functions (of the variable $z$) with integer coefficients.
\end{theorem}
Here, by a quadratic irrational we mean an algebraic real number of strict degree two, 
so the simple continued fraction expansion of $\theta$ is eventually periodic.
Theorem~\eqref{thm:main-gen_fn_rat} generalises the known case when 
$\theta$ has pre-periodic continued fraction expansion 
of period one. We also note that the generating functions $F$ and $G$ 
are rational also follows from an argument in~\cite{ShallitLenstra2001}. 
However, the techniques we present here are different.
In fact, we obtain explicit formulas for $F$ and $G$. 
Namely, if 
$\ell$ denotes the eventual period of the continued fraction expansion of $\theta$, 
and $\kay$ denotes any pre-period of the continued fraction expansion of $\theta$
then we may write
\begin{align}
F(z)&\;=\;\sum_{0\leq n <\kay}p_nz^n+\sum_{\kay\leq n<\kay+\ell} F_{n}(z)
\end{align}
and
\begin{align}
G(z)&\;=\;\sum_{0\leq n <\kay}q_nz^n+\sum_{\kay\leq n<\kay+\ell} G_{n}(z)
\end{align}
where, for each $n\in\mathbb{N}_0$,
\begin{align}
F_{n}(z)\;=\;\sum_{r\in\mathbb{N}_0} p_{n+r\ell}z^{n+r\ell} \qquad \mbox{and} \qquad
G_{n}(z)\;=\;\sum_{r\in\mathbb{N}_0} q_{n+r\ell}z^{n+r\ell}
\end{align}
With this notation, the above Theorem~\ref{thm:main-gen_fn_rat} is a corollary of the following result.
\begin{theorem}\label{thm:main-gen_fn_rat2}
Let 
$\theta\in\mathbb{R}\setminus\mathbb{Q}$ 
be a quadratic irrational.
Let $\ell$ denote the minimal (eventual) period of the simple continued fraction expansion of $\theta$, and 
let $\kay\geq 2$ be any non-minimal pre-period.
For each $n$, let $p_n/q_n$ denote the $n$th best rational approximant of $\theta$.
For $\kay\leq n<\kay+\ell$, let $F_n$ and $G_n$ be defined as above.
Then
\begin{align}
F_{n}(z)&\;=\;\frac{z^n p_n+z^{n+\ell} (-1)^{\ell+1}p_{n-\ell}}{1-(-1)^\kay \delta z^\ell+(-1)^\ell z^{2\ell}}\label{eq:Fn}
\end{align}
and
\begin{align}
G_{n}(z)&\;=\;\frac{z^n q_n+z^{n+\ell} (-1)^{\ell+1}q_{n-\ell}}{1-(-1)^\kay \delta z^\ell+(-1)^\ell z^{2\ell}}\label{eq:Gn}
\end{align}
where
\begin{equation}
\delta\;=\;
\delta_\theta\;=\;
q_{\kay}p_{\kay+\ell-1}-p_{\kay}q_{\kay+\ell-1}-q_{\kay-1}p_{\kay+\ell}+p_{\kay-1}q_{\kay+\ell} \ .
\end{equation}
\end{theorem}

Using our result, we can easily derive the 
following results concerning the 
{\it L{\'e}vy constant}
of a quadratic irrational.
Recall that, given a real number $\theta$ 
with $n$th best rational approximant given by $p_n/q_n$ for each $n$, 
the L{\'e}vy constant of $\theta$, when it exists, 
is given by the following expression
\begin{equation}
\beta(\theta)\;=\;\lim_{n\to\infty}\frac{1}{n}\log q_n \ .
\end{equation}
Paul L{\'e}vy~\cite{Levy1929} showed, following earlier work by A.~Ya.~Khintchine, 
that 
\begin{equation}
\beta(\theta)\;=\;\frac{\pi^2}{12\log 2} \qquad \mbox{for Lebesgue-almost every} \ \theta \ .
\end{equation}
(See~\cite{LevyBook,KhinchinBook,Lehmer1939} for more details.)
\begin{comment}
C.~Faivre~\cite{Faivre1992,Faivre1997} showed that for 
any real number $\beta$ in the interval $[\frac{1+\sqrt{5}}{2},\infty)$ there exists 
a real number $\theta$ such that
$\beta(\theta)=\beta$.
\end{comment}
\begin{comment}
A straightforward comparison with the golden mean, 
together with the fact that $q_n$ increases with any of the partial quotients $a_1,a_2,\ldots,a_n$, implies that
\begin{equation}
q_n^{1/n}\gtrsim \frac{1+\sqrt{5}}{2}
\end{equation}
\end{comment}
It was shown by Jager and Liardet~\cite{JagerLiardet1988} 
that for every quadratic irrational, the L{\'e}vy constant exists. 
As an immediate  corollary to Theorem~\eqref{thm:main-gen_fn_rat} above 
we get a new proof of the following result, 
which was implicitly contained in~\cite{JagerLiardet1988}.
\begin{theorem}\label{thm:Levy_const-vs-spec_radius}
Let $\theta\in\mathbb{R}\setminus\mathbb{Q}$ be a quadratic irrational. 
Let $\ell$ denote the (eventual) period of the simple continued fraction expansion of $\theta$.
Let $M_\theta$ denote the element of $\mathrm{PSL}(2,\mathbb{Z})$ corresponding to the simple continued fraction expansion of $\theta$.
Then 
\begin{equation}
\beta(\theta)\;=\;\frac{1}{\ell}\log\rad(M_\theta)
\end{equation}
where, for every $A\in\mathrm{PSL}(2,\mathbb{C})$, 
we denote by $\rad(A)$ the spectral radius of either of the linear transformations corresponding to $A$. 
\end{theorem}
(The description of $M_\theta$ will be given in more detail below.)
\begin{comment}
As a corollary we get the following result originally due to J.~Wu.
\begin{corollary}
The set of $\beta(\theta)$, where $\theta$ is a quadratic irrational, is dense in $[\frac{1+\sqrt{5}}{2},\infty)$.
\end{corollary}
\end{comment}
\begin{comment}
V. I.~Arnold~\cite[Problem 1993-11B]{ArnoldsProblems} discussed the problem 
of Gauss-Kuzmin statistics for rationals and quadratic irrationals.
\end{comment}

\subsection{Notation and Terminology.}
Let $\mathbb{N}$ and $\mathbb{N}_0$ denote the set of positive and non-negative integers respectively.
Denote the set of integers by $\mathbb{Z}$ and, for each positive integer $\ell$ let $\mathbb{Z}_\ell$ 
denote the set of integers modulo $\ell$, i.e. $\mathbb{Z}_\ell=\mathbb{Z}/\ell\mathbb{Z}$.
Let $\mathbb{R}$ and $\mathbb{C}$ denote, as usual, the real and complex number fields.
Given an arbitrary polynomial $\eta$, over either $\mathbb{R}$ or $\mathbb{C}$, we denote the discriminant by $\discr_\eta$. 
  
Given $\theta\in\mathbb{R}$, 
let $\lfloor\theta\rfloor$ denote the {\it integer part} of $\theta$, 
{\it i.e.\/}, greatest integer less than or equal to $\theta$, 
and let $\{\theta\}=\theta-\lfloor\theta\rfloor$ denote the {\it fractional part} of $\theta$. 

Let $\mathcal{S}_\ell$ denote the symmetric group of size $\ell$, {\it i.e.}, 
the permutation group on a set of $\ell$ elements.
We will denote the signature of a permutation $\upsilon$ in $\mathcal{S}_\ell$ by $\sgn(\upsilon)$.
(Recall, any permutation can be written, non-uniquely,  
as a product of adjacent transpositions $( s \ s+1 )$, and the parity of the total number of such transpositions is the signature.)
Such a permutation $\upsilon$ can be expressed as a product of cycles.
A single cycle will be expressed as 
$(s_{1},s_{2},\ldots,s_{k})$, 
for some $s_1,s_2,\ldots,s_k\in \{1,2,\ldots,\ell\}$ 
where 
$\upsilon (s_{j})=s_{j-1}$ for each $j$, where addition is taken mod $\ell$.
We denote the cardinality of a set $S$ by $\card S$. 

\subsection*{Acknowledgements.}
The authors would like to thank the Mathematics Department at Uppsala University and IME-USP for their continued hospitality and support.
We thank Charles Tresser and Edson de Faria for reading an earlier draft of this article and 
Alby Fisher, Sinai Robins and Andreas Str\"ombergsson for their questions and comments.
We also thank Alex Har\'o for informing us of a mistake in the previous draft and Jarkko Peltom\"aki, 
for bringing the article~\cite{ShallitLenstra2001} to our attention.

\section{Preliminaries.}
\subsection{Continued fractions.}\label{sect:ctd_frac_1}
In this section we recall some basic properties of simple continued fraction expansions.
The main aim is to set up notation for the rest of the paper.
For more details we recommend that the reader consults~\cite{KhinchinBook,HardyWrightBook}. 
\subsubsection{Best rational approximants.}\label{subsect:best_approx}
Let $\theta\in[0,1]\setminus\mathbb{Q}$.
We note that the discussion below can be carried out also for irrational points outside $[0,1]$, with suitable modifications.
However, for simplicity we restrict ourselves to the case $\theta\in [0,1]\setminus\mathbb{Q}$.

The simple continued fraction expansion of $\theta$ is denoted by 
\begin{equation}
\theta
\;=\;\left[a_1,a_2,\ldots\right]
\;=\;\frac{1}{a_1+}\frac{1}{a_2+}\cdots\frac{1}{a_n+}\cdots \ ,
\end{equation}
where $a_1,a_2,\ldots$ are positive integers called the {\it partial quotients} of the continued fraction.
Define the {\it $n$th convergent} of $\theta$ to be
\begin{equation}\label{eq:best_rat_approx}
[a_1,a_2,\ldots,a_n]
\;=\;\frac{1}{a_1+}\frac{1}{a_2+}\cdots\frac{1}{a_{n-1}+}\frac{1}{a_n} \ .
\end{equation}
This is a rational number which we will express as $p_n/q_n$, 
where $p_n$ and $q_n$ are positive integers having no common factors.
The following property is satisfied for all $n\in\mathbb{N}$
\begin{equation}
\left|\theta-\frac{p_n}{q_n}\right|
\;\leq\; 
\inf_{\frac{p}{q}\in\mathbb{Q}: q\leq q_n}\left|\theta-\frac{p}{q}\right| \ .
\end{equation}
For this reason $p_n/q_n$ is also called the {\it $n$th best rational approximant} of $\theta$.
Identifying $\mathbb{C}$ with $\mathbb{P}(\mathbb{C}^2)\setminus \left\{[1:0]\right\}$, 
equation~\eqref{eq:best_rat_approx} can be expressed in matrix form as
\begin{align}
\left[\begin{array}{c}p_n\\ q_n\end{array}\right]
&\;=\;
\left[\begin{array}{cc}0&1\\ 1&a_1\end{array}\right]
\left[\begin{array}{cc}0&1\\ 1&a_2\end{array}\right]
\cdots
\left[\begin{array}{cc}0&1\\ 1&a_{n-1}\end{array}\right]
\left[\begin{array}{c}1\\ a_{n}\end{array}\right] \ . \label{eq:pq_n}
\end{align}
Similarly
\begin{align}
\left[\begin{array}{c}p_{n-1}\\ q_{n-1}\end{array}\right]
&\;=\;
\left[\begin{array}{cc}0&1\\ 1&a_1\end{array}\right]
\left[\begin{array}{cc}0&1\\ 1&a_2\end{array}\right]
\cdots
\left[\begin{array}{cc}0&1\\ 1&a_{n-1}\end{array}\right]
\left[\begin{array}{c}0\\ 1\end{array}\right] \ . \label{eq:pq_n-1}
\end{align}
Thus combining~\eqref{eq:pq_n} and~\eqref{eq:pq_n-1} we find that
\begin{equation}\label{eq:recursion_rel_simple}
\left[\begin{array}{cc}p_{n-1}&p_{n}\\ q_{n-1}&q_{n}\end{array}\right]
\;=\;
\left[\begin{array}{cc}0&1\\ 1&a_1\end{array}\right]
\left[\begin{array}{cc}0&1\\ 1&a_2\end{array}\right]
\cdots
\left[\begin{array}{cc}0&1\\ 1&a_{n-1}\end{array}\right]
\left[\begin{array}{cc}0&1\\ 1&a_{n}\end{array}\right] \ .
\end{equation}
We therefore inductively get the following recurrence relation
\begin{equation}
\left[\begin{array}{cc}p_{n-1}& p_{n}\\ q_{n-1}&q_{n}\end{array}\right]
\;=\;
\left[\begin{array}{cc}p_{n-2} & p_{n-1}\\ q_{n-2}&q_{n-1}\end{array}\right]
\left[\begin{array}{cc}0&1\\ 1&a_{n}\end{array}\right] \ ,
\quad
\left[\begin{array}{cc}p_{0}&p_{1}\\ q_{0}&q_{1}\end{array}\right]
\;=\;
\left[\begin{array}{cc}0&1\\ 1&a_1\end{array}\right] \ ,
\end{equation}
and, more generally, for any non-negative integer $m\leq n-2$,
\begin{equation}
\label{eq:recursion_rel}
\left[\begin{array}{cc}p_{n-1}& p_{n}\\ q_{n-1}&q_{n}\end{array}\right]
\;=\;
\left[\begin{array}{cc}p_{n-m-2} & p_{n-m-1}\\ q_{n-m-2}&q_{n-m-1}\end{array}\right]
\left[\begin{array}{cc}0&1\\ 1&a_{n-m}\end{array}\right]
\cdots
\left[\begin{array}{cc}0&1\\ 1&a_{n}\end{array}\right] \ .
\end{equation}
For each suitable $n$ and $m$ define
\begin{equation}
\label{eq:AB_definition}
\left[\begin{array}{cc}B^{(m)}_{n-m-1}&B^{(m+1)}_{n-m-1}\\ A^{(m)}_{n-m-1}&A^{(m+1)}_{n-m-1}\end{array}\right]
\;=\;
\left[\begin{array}{cc}0&1\\ 1&a_{n-m}\end{array}\right]
\cdots
\left[\begin{array}{cc}0&1\\ 1&a_{n}\end{array}\right] \ .
\end{equation}
Observe that the 
$A^{(m)}_{n-m}$ and $B^{(m)}_{n-m}$ 
are well-defined since we have the relation
\begin{equation}\label{eq:A_relation_1}
\left[\begin{array}{ll}
B^{(m)}_{n-m-1}& B^{(m+1)}_{n-m-1}\\
A^{(m)}_{n-m-1}& A^{(m+1)}_{n-m-1}
\end{array}\right]
\;=\;
\left[\begin{array}{cc}
0&1\\1&a_{n-m}
\end{array}\right]
\left[\begin{array}{ll}
B^{(m-1)}_{n-m}& B^{(m)}_{n-m}\\
A^{(m-1)}_{n-m}& A^{(m)}_{n-m}
\end{array}\right] \ .
\end{equation}
This could equally be inferred from the corresponding dual relation
\begin{equation}\label{eq:A_relation_2}
\left[\begin{array}{ll}
B^{(m)}_{n-m-1}& B^{(m+1)}_{n-m-1}\\
A^{(m)}_{n-m-1}& A^{(m+1)}_{n-m-1}
\end{array}\right]
\;=\;
\left[\begin{array}{ll}
B^{(m-1)}_{n-m-1}& B^{(m)}_{n-m-1}\\
A^{(m-1)}_{n-m-1}& A^{(m)}_{n-m-1}
\end{array}\right]
\left[\begin{array}{cc}
0&1\\1&a_{n}
\end{array}\right] \ .
\end{equation}
We gather some well-known basic properties in the following proposition.
\begin{proposition}\label{prop:basic_properties}
Let $\theta\in[0,1]\setminus\mathbb{Q}$.
Then the best rational approximants $p_n/q_n$ satisfy the following:
\begin{enumerate}
\item
$p_{n-1}q_{n}-p_{n}q_{n-1}=(-1)^n$
[Lagrange identity];
\item
$\frac{p_{n-1}}{p_{n}}=[a_{n},a_{n-1},\ldots,a_3,a_2]$;
\item
$\frac{q_{n-1}}{q_{n}}=[a_{n},a_{n-1},\ldots,a_2,a_1]$;
\item
$\frac{p_{n-1}z+p_n}{q_{n-1}z+q_n}=[a_1,a_2,\ldots,a_n,1/z]$.
\end{enumerate}
\end{proposition}
\begin{proof}
The first item follows by taking the determinant of
~\eqref{eq:recursion_rel_simple}.
The second and third items follow by taking the transpose of
~\eqref{eq:recursion_rel_simple}.
The last item follows by applying~\eqref{eq:recursion_rel_simple} 
to the vector 
$\left[\begin{array}{c}z\\ 1\end{array}\right]$.
\end{proof}
To summarize, the numerators and denominators of the convergents satisfy 
a recursion relation~\eqref{eq:recursion_rel_simple} which 
may be (re)stated as
\begin{align}
p_{n}&\;=\;a_{n}p_{n-1}+p_{n-2} \ ;& p_0&\;=\;0& p_1&\;=\;1\label{eq:recurrence_rel_p}\\
q_{n}&\;=\;a_{n}q_{n-1}+q_{n-2} \ ;& q_0&\;=\;1& q_1&\;=\;a_1\label{eq:recurrence_rel_q}
\end{align}
and more generally
\begin{equation}\label{eq:AB_expansion}
\left.\begin{array}{lcllcl}
p_{n}&=& A^{(1)}_{n-1}p_{n-1}+B^{(1)}_{n-1}p_{n-2} \qquad &
q_{n}&=& A^{(1)}_{n-1}q_{n-1}+B^{(1)}_{n-1}q_{n-2}\\
     &=& A^{(2)}_{n-2}p_{n-2}+B^{(2)}_{n-2}p_{n-3} \qquad &
     &=& A^{(2)}_{n-2}q_{n-2}+B^{(2)}_{n-2}q_{n-3}\\
&\vdots& &
&\vdots&
\end{array}\right.
\end{equation}
where $A^{(m)}_{n-m}$ and $B^{(m)}_{n-m}$ are 
non-negative integers satisfying the recurrence relations
\begin{equation}
\label{eq:AB_recurrence_rel}
\begin{gathered}
B^{(m)}_{n-m}
\;=\;A^{(m-1)}_{n-m+1} \qquad
A^{(m)}_{n-m}
\;=\;B^{(m-1)}_{n-m+1}+a_{n-m+1}A^{(m-1)}_{n-m+1}\\
B^{(0)}_{n}\;=\;0 \qquad 
B^{(1)}_{n-1}\;=\;1\;=\;A^{(0)}_{n}
\end{gathered}
\ \ .
\end{equation}
In fact, we have the following.
\begin{proposition}
Let $\theta\in[0,1]\setminus\mathbb{Q}$.
Then the coefficients $A^{(m)}_{n}$ satisfy the following:
\begin{align}
A^{(m)}_{n-m}&\;=\;a_{n-m+1}A^{(m-1)}_{n-m+1}+A^{(m-2)}_{n-m+2}\label{eq:AB_splitting_rel_1}\\
A^{(m)}_{n-m}&\;=\;a_n A^{(m-1)}_{n-m}+A^{(m-2)}_{n-m}\label{eq:AB_splitting_rel_2}
\end{align}
\end{proposition}
\begin{proof}
The first equality is just an application of~\eqref{eq:AB_recurrence_rel}.
For the second equality, 
multiply out the relation~\eqref{eq:A_relation_2} 
and apply~\eqref{eq:AB_recurrence_rel}
\end{proof}
\begin{proposition}\label{prop:pA&qA_splitting_rel} 
Let $\theta\in[0,1]\setminus\mathbb{Q}$.
Then the best rational approximants $p_n/q_n$ satisfy the following:
\begin{align}
           p_{n+m+1}&\;=\;p_{n}A^{(m+1)}_{n}+p_{n-1}A^{(m)}_{n+1}\label{eq:pA_splitting_rel_1}\\
(-1)^{m+1} p_{n-m-1}&\;=\;p_{n}A^{(m-1)}_{n-m}-p_{n-1}A^{(m)}_{n-m}\label{eq:pA_splitting_rel_2}\\
           q_{n+m+1}&\;=\;q_{n}A^{(m+1)}_{n}+q_{n-1}A^{(m)}_{n+1}\label{eq:qA_splitting_rel_1}\\
(-1)^{m+1} q_{n-m-1}&\;=\;q_{n}A^{(m-1)}_{n-m}-q_{n-1}A^{(m)}_{n-m}\label{eq:qA_splitting_rel_2}
\end{align}
\end{proposition}
\begin{proof}
The first and third equality follows by applying the expression 
for $B^{(m)}_{n-m}$ in~\eqref{eq:AB_recurrence_rel} 
to~\eqref{eq:AB_expansion}.
For the second and fourth equality, equations~\eqref{eq:recursion_rel},~\eqref{eq:AB_definition} and~\eqref{eq:AB_recurrence_rel}, imply that
\begin{equation}\label{eq:pq_AB_rel_1}
\left[\begin{array}{cc}p_{n-1}&p_{n}\\ q_{n-1}&q_{n}\end{array}\right]
\;=\;
\left[\begin{array}{cc}p_{n-m-2}&p_{n-m-1}\\ q_{n-m-2}&q_{n-m-1}\end{array}\right]
\left[\begin{array}{cc}A^{(m-1)}_{n-m}& A^{(m)}_{n-m}\\ A^{(m)}_{n-m-1}&A^{(m+1)}_{n-m-1}\end{array}\right] \ .
\end{equation}
The rightmost matrix factor on the right-hand side has determinant $(-1)^{m+1}$, 
as the matrix is the product of $m+1$ matrices of determinant $-1$. 
Thus, applying the inverse of the rightmost matrix to both sides and multiplying out gives the required equalities. 
\end{proof}
\begin{proposition}\label{prop:trace_AB_pq}
Let $\theta\in[0,1]\setminus\mathbb{Q}$.
Then the best rational approximants $p_n/q_n$ satisfy the following:
\begin{align}\label{eq:trace_AB_pq}
&A^{(m-1)}_{n-m}+A^{(m+1)}_{n-m-1}\notag\\
&\;=\;(-1)^{n-m-1}\left(q_{n-m-1}p_{n-1}-p_{n-m-1}q_{n-1}-q_{n-m-2}p_{n}+p_{n-m-2}q_{n}\right)
\end{align}
\end{proposition}
\begin{proof}
Rearranging equation~\eqref{eq:pq_AB_rel_1} from the preceding proposition by applying an appropriate inverse 
and applying Proposition~\ref{prop:basic_properties}(1) to compute the necessary determinant gives
\begin{equation}
\left[\begin{array}{cc}A^{(m-1)}_{n-m}& A^{(m)}_{n-m}\\ A^{(m)}_{n-m-1}&A^{(m+1)}_{n-m-1}\end{array}\right]
\;=\;
(-1)^{n-m-1}
\left[\begin{array}{cc}q_{n-m-1}&-p_{n-m-1}\\ -q_{n-m-2}&p_{n-m-2}\end{array}\right]
\left[\begin{array}{cc}p_{n-1}&p_{n}\\ q_{n-1}&q_{n}\end{array}\right] \ .
\end{equation}
Taking the trace of both sides now gives the equation~\eqref{eq:trace_AB_pq}, as required.
\end{proof}

%
\subsubsection{Quadratic Irrationals.}\label{subsubsect:Quad_Irr}
Let $\theta\in[0,1]\setminus\mathbb{Q}$ be a quadratic irrational.
By this we will mean that $\theta$ is an algebraic number 
with minimal polynomial $\chi$ with (strict) degree two.
A theorem of Lagrange~\cite[p. 56]{KhinchinBook} implies 
that $\theta$ has a pre-periodic simple continued fraction expansion.
Hence, there exist positive integers 
$a_1,a_2,\ldots,a_\kay,\ldots,a_{\kay+\ell}$ 
such that
\begin{equation}
\theta\;=\;[a_1,a_2,\ldots,a_{\kay},\overline{a_{\kay+1},\ldots,a_{\kay+\ell}}] \ .
\end{equation}
We call the minimal such $\ell$ the {\it period}. 
We call any such $\kay$ a {\it pre-period} and the least such pre-period the 
{\it minimal pre-period} of the simple continued fraction expansion.
\begin{remark}
Observe that, since $a_{n+\ell}=a_{n}$ whenever $n>\kay$, it follows from~\eqref{eq:AB_definition} 
that, for all non-negative integers $m$ and $n$ satisfying $n-m>\kay$, we have
\begin{equation}
A^{(m)}_{n-m+\ell-1}\;=\;A^{(m)}_{n-m-1} \ , \qquad B^{(m)}_{n-m+\ell-1}\;=\;B^{(m)}_{n-m-1} \ .
\end{equation}
\begin{comment}
This follows as
\begin{align}
\left[\begin{array}{cc}B^{(m)}_{n-m-1}&B^{(m+1)}_{n-m-1}\\ A^{(m)}_{n-m-1}&A^{(m)}_{n-m-1}\end{array}\right]
&=
\left[\begin{array}{cc}0&1\\ 1&a_{n-m}\end{array}\right]
\left[\begin{array}{cc}0&1\\ 1&a_{n-m+1}\end{array}\right]
\cdots
\left[\begin{array}{cc}0&1\\ 1&a_{n}\end{array}\right]\\
&=
\left[\begin{array}{cc}0&1\\ 1&a_{n-m+\ell}\end{array}\right]
\left[\begin{array}{cc}0&1\\ 1&a_{n-m+1+\ell}\end{array}\right]
\cdots
\left[\begin{array}{cc}0&1\\ 1&a_{n+\ell}\end{array}\right]\\
&=
\left[\begin{array}{cc}B^{(m)}_{n+\ell-m-1}&B^{(m+1)}_{n+\ell-m-1}\\ A^{(m)}_{n+\ell-m-1}&A^{(m)}_{n+\ell-m-1}\end{array}\right]
\end{align}
\end{comment}
Therefore the following quantities are well-defined
\begin{equation}\label{eq:AB_periodic}
A^{(m)}_{(n)}\;=\;A^{(m)}_{\kay+j} \ , \qquad 
B^{(m)}_{(n)}\;=\;B^{(m)}_{\kay+j} \ , \qquad \mbox{for any} \ j\geq 0, \ \kay+j=n\mod \ell \ .
\end{equation}
\end{remark}
Let
$\theta^{\caret \kay}=T^{\kay}(\theta)$, 
where $T$ denotes the Gauss transformation on the interval $[0,1]$. 
Then $T^\ell(\theta^{\caret \kay})=\theta^{\caret \kay}$.
Thus $\theta^{\caret \kay}$ is a solution to the equation
\begin{equation}
\theta^{\caret \kay}
\;=\;
\frac{1}{a_{\kay+1}+}\frac{1}{a_{\kay+2}+}\cdots\frac{1}{a_{\kay+\ell}+\theta^{\caret \kay}}
\end{equation}
and $\theta$ can be expressed as
\begin{equation}
\theta
\;=\;
\frac{1}{a_{1}+}\frac{1}{a_{2}+}\cdots\frac{1}{a_{\kay}+\theta^{\caret \kay}} \ .
\end{equation}
These can be written in matrix form as
\begin{equation}
\left[\begin{array}{c}\theta^{\caret \kay}\\ 1\end{array}\right]
\;=\;
\left[\begin{array}{cc}0&1\\ 1&a_{\kay+1}\end{array}\right]
\left[\begin{array}{cc}0&1\\ 1&a_{\kay+2}\end{array}\right]
\cdots
\left[\begin{array}{cc}0&1\\ 1&a_{\kay+\ell}\end{array}\right]
\left[\begin{array}{c}\theta^{\caret \kay}\\ 1\end{array}\right]
\end{equation}
and
\begin{equation}
\left[\begin{array}{c}\theta\\ 1\end{array}\right]
\;=\;
\left[\begin{array}{cc}0&1\\ 1&a_{1}\end{array}\right]
\left[\begin{array}{cc}0&1\\ 1&a_{2}\end{array}\right]
\cdots
\left[\begin{array}{cc}0&1\\ 1&a_{\kay}\end{array}\right]
\left[\begin{array}{c}\theta^{\caret \kay}\\ 1\end{array}\right] \ ,
\end{equation}
where we identify matrices with their corresponding linear fractional transformations.
Let 
\begin{equation}
M_1
\;=\;
\left[\begin{array}{cc}0&1\\ 1&a_{\kay+1}\end{array}\right]
\left[\begin{array}{cc}0&1\\ 1&a_{\kay+2}\end{array}\right]
\cdots
\left[\begin{array}{cc}0&1\\ 1&a_{\kay+\ell}\end{array}\right]
\end{equation}
and
\begin{equation}
M_0
\;=\;
\left[\begin{array}{cc}0&1\\ 1&a_{1}\end{array}\right]
\left[\begin{array}{cc}0&1\\ 1&a_{2}\end{array}\right]
\cdots
\left[\begin{array}{cc}0&1\\ 1&a_{\kay}\end{array}\right] \ .
\end{equation}
Let $M_\theta=M_0 M_1 M_0^{-1}$.
We call this the element of $\mathrm{PSL}(2,\mathbb{Z})$ 
{\it corresponding to continued fraction expansion} of $\theta$.
We call $M_1$ the element of $\mathrm{PSL}(2,\mathbb{Z})$ 
{\it corresponding to the periodic part of the continued fraction expansion} of $\theta$.
Then 
$\theta$ is a solution of the fixed point equation $M_\theta(z)=z$,
$\theta^{\caret\kay}$ is a solution of the fixed point equation $M_1(z)=z$,
and 
$\theta=M_0(\theta^{\caret\kay})$.
From 
equation~\eqref{eq:AB_definition}, 
equation~\eqref{eq:AB_recurrence_rel}, 
and applying 
definition~\eqref{eq:AB_periodic} 
we have
\begin{equation}\label{eq:a1a2...al}
M_1
\;=\;
\left[\begin{array}{cc}
B^{(\ell-1)}_{(\kay)} & B^{(\ell)}_{(\kay)}\\A^{(\ell-1)}_{(\kay)} & A^{(\ell)}_{(\kay)}
\end{array}\right]
\;=\;
\left[\begin{array}{cc}
A^{(\ell-2)}_{(\kay+1)} & A^{(\ell-1)}_{(\kay+1)}\\A^{(\ell-1)}_{(\kay)} & A^{(\ell)}_{(\kay)}
\end{array}\right]
\end{equation}
and
\begin{equation}
M_0
\;=\;
\left[\begin{array}{cc}
B^{(\kay-1)}_{(0)} & B^{(\kay)}_{(0)}\\A^{(\kay-1)}_{(0)} & A^{(\kay)}_{(0)}
\end{array}\right]
\;=\;
\left[\begin{array}{cc}
A^{(\kay-2)}_{(1)} & A^{(\kay-1)}_{(1)}\\A^{(\kay-1)}_{(0)} & A^{(\kay)}_{(0)}
\end{array}\right] \ .
\end{equation}
To $M_1$, or equivalently to $\theta^{\caret\kay}$, 
there are two degree two polynomials which are naturally associated with it:
\begin{align}
\chi_{\theta^{\caret\kay}}(z)
&=
z^2-\left( A^{(\ell-2)}_{(\kay+1)}+A^{(\ell)}_{(\kay)} \right)z+(-1)^\ell\\
\omega_{\theta^{\caret\kay}}(z)
&=
A^{(\ell-1)}_{(\kay)}z^2 +\left(A^{(\ell)}_{(\kay)}-A^{(\ell-2)}_{(\kay+1)}\right)z - A^{(\ell-1)}_{(\kay+1)} \ .
\end{align}
The first is just the characteristic polynomial of $M_1$, 
and thus the roots are the eigenvalues of $M_1$. 
(Observe that $\det M_1=(-1)^\ell$ as $M_1$ is a product of $\ell$ matrices with determinant $-1$.)
The second is the minimal polynomial of $\theta^{\caret\kay}$, 
and thus the roots correspond to the eigenvectors of $M_1$.
Notice that these two polynomials have the same discriminant, and thus the roots are rationally related.

In the same way we may associate a pair of degree two polynomials $\chi_\theta$ and $\omega_\theta$ to $M_\theta$ 
(as the defining polynomials for eigenvalues and eigenvectors).
Since the determinant and trace are conjugacy invariant we find that 
\begin{equation}
\chi_{\theta}(z)\;=\;\chi_{\theta^{\caret\kay}}(z) \ .
\end{equation}
Also observe that the roots of $\omega_\theta$ are the image under $M_0$ of the roots of $\omega_{\theta^{\caret\kay}}$.
Thus it follows 
(noticing that 
$\left(A^{(\kay-2)}_{(1)}-A^{(\kay-1)}_{(0)}z\right)$ 
is the denominator of $M_0^{-1}$) 
that
\begin{equation}
\omega_{\theta}(z)\;=\;\left(A^{(\kay-2)}_{(1)}-A^{(\kay-1)}_{(0)}z\right)^2\omega_{\theta^{\caret\kay}}(M_0^{-1}(z)) \ .
\end{equation}
\section{Generating functions.}\label{sect:Gen_fn}
\subsection{Generating functions associated to a simple continued fraction expansion.}\label{subsect:Gen_fn-ctd_frac}
Given $\theta\in[0,1]\setminus\mathbb{Q}$, denote by $p_n/q_n$ the $n$-th best rational approximant.
Consider the generating functions
\begin{equation}
F(z)\;=\;\sum_{n\geq 0}p_nz^n 
\qquad \mbox{and} \qquad 
G(z)\;=\;\sum_{n\geq 0}q_nz^n \ .
\end{equation}
The aim of this section is to prove Theorem~\ref{thm:main-gen_fn_rat}.
Assume that $\theta$ is a quadratic irrational.
Then, for some positive integers $a_1,a_2,\ldots,a_{\kay+\ell}\in\mathbb{N}$,
\begin{equation}
\theta\;=\;\left[a_1,a_2,\ldots,a_\kay,\overline{a_{\kay+1},\ldots,a_{\kay+\ell}}\right] \ .
\end{equation}
In the first part of our construction $\ell$ may be an arbitrary period, 
$\kay$ may be an arbitrary pre-period.
However, the second part will require $\ell$ to be the minimal period 
and $\kay$ to {\it not} be the minimal pre-period of the continued fraction expansion
(though this is only really necessary in the periodic case, {\it i.e.} when $\kay=0$).
Let us set up the following notation.
\noindent
For $m\in\mathbb{N}$, define
\begin{equation}
I_m(z)
\;=\;
\sum_{0\leq j<m}p_jz^j \ , 
\qquad 
J_m(z)
\;=\;
\sum_{0\leq j<m}q_jz^j \ ,
\end{equation}
and also define $I_0(z)=0=J_0(z)$.
For $n\in\mathbb{N}_0$, define
\begin{equation}\label{def:Fn+Gn}
F_{n}(z)
\;=\;
\sum_{r\in\mathbb{N}_0}p_{n+r\ell}z^{n+r\ell} 
\qquad \mbox{and} \qquad
G_{n}(z)
\;=\;
\sum_{r\in\mathbb{N}_0}q_{n+r\ell}z^{n+r\ell} \ .
\end{equation}
First consider the generating function $F(z)$.
Observe that
\begin{equation}\label{eq:F_splitting-k=0}
F(z)
\;=\;
I_{\kay}(z)+\sum_{\kay\leq n<\kay+\ell}F_{n}(z) \ .
\end{equation}
More generally, for any $m\in\mathbb{N}_0$, 
we may split the generating function $F(z)$ as
\begin{align}
F(z)
&\;=\;
I_{\kay+m+1}(z)
+
\sum_{\kay\leq n<\kay+\ell}
F_{m+1+n}(z) \ .
\label{eq:F_splitting}
\end{align}
We will now use the recurrence relations~\eqref{eq:AB_expansion} 
to express the right-most summation in terms of 
$F_\kay, F_{\kay+1},\ldots,F_{\kay+\ell-1}$.
Fix an integer
$m$, $1\leq m<\ell+1$, and take
$n\in\mathbb{N}_0$.
Applying the recurrence relations~\eqref{eq:AB_expansion} 
(which we can do provided that $m\geq 1$)
followed by the periodicity property~\eqref{eq:AB_periodic}
\begin{align}
F_{m+1+n}(z)
&=
\sum_{r\in\mathbb{N}_0} p_{m+1+n+r\ell} \, z^{m+1+n+r\ell}\\
&=
\sum_{r\in\mathbb{N}_0}
\left(A^{(m)}_{n+1+r\ell} \, p_{n+1+r\ell}+B^{(m)}_{n+1+r\ell} \, p_{n+r\ell}\right)z^{m+1+n+r\ell}\\
&=
A^{(m)}_{(n+1)}z^{m}
F_{n+1}(z)
+
B^{(m)}_{(n+1)}z^{m+1}
F_n(z) \ .
\end{align}
This relation holds for arbitrary $n\in\mathbb{N}_0$.
Now restrict to the case $\kay\leq n<\kay+\ell$.
All the $F_n$'s in the above summation are one of the
$F_{\kay}, F_{\kay+1},\ldots,F_{\kay+\ell-1}$, 
except for $F_{\kay+\ell}$ which satisfies 
$F_{\kay+\ell}(z)=F_{\kay}(z)-p_\kay z^\kay$.
Hence equation~\eqref{eq:F_splitting} becomes
\begin{align}
F(z)
&\;=\;
I_{\kay+m+1}(z)
-A^{(m)}_{(\kay+\ell)}p_{\kay} z^{\kay+m}\notag\\
& \ \ 
+\sum_{\kay\leq n<\kay+\ell}
\left(A^{(m)}_{(n+1)}z^m F_{n+1}(z)+ B^{(m)}_{(m+1)}z^{m+1} F_m(z)\right) \ .
\end{align}
By~\eqref{eq:AB_periodic}, 
$A^{(m)}_{(n+1)}=A^{(m)}_{(n+\ell+1)}$ and 
$B^{(m)}_{(n+1)}=B^{(m)}_{(n+\ell+1)}$,
so the expression on the right-hand side of the 
equation above may be rewritten to give
\begin{align}
F(z)&\;=\;
I_{\kay+m+1}(z)
-A^{(m)}_{(\kay+\ell)}p_{\kay}z^{\kay+m}\notag\\
& \ \ 
+\sum_{\kay\leq n<\kay+\ell}
\left(A^{(m)}_{(n+\ell)}z^m+B^{(m)}_{(n+\ell+1)}z^{\kay+1}\right)F_n(z) \ .
\end{align}
For $1\leq m<\ell+1$, substituting into the 
expression~\eqref{eq:F_splitting-k=0} and rearranging therefore gives
\begin{align}\label{eq:F_relation}
&
\sum_{\kay\leq n<\kay+\ell}
\left(1-A^{(m)}_{(n+\ell)}z^m-B^{(m)}_{(n+\ell+1)}z^{m+1}\right)
F_n(z)\notag\\
&\;=\;
I_{\kay+m+1}(z)-I_{\kay}(z)-A^{(m)}_{(\kay+\ell)}p_{\kay}z^{\kay+m} \ .
\end{align}
For $1\leq m<\ell+1$, exactly the same argument as for $F(z)$ also yields the following relation for $G(z)$,
\begin{align}\label{eq:G_relation}
&
\sum_{\kay\leq n<\kay+\ell}
\left(1-A^{(m)}_{(n+\ell)}z^m-B^{(m)}_{(n+\ell+1)}z^{m+1}\right)
G_n(z)\notag\\
&\;=\;
I_{\kay+m+1}(z)-I_{\kay}(z)-A^{(m)}_{(\kay+\ell)}q_{\kay}z^{\kay+m} \ .
\end{align}
(In fact, the generating function of any sequence satisfying 
the recurrence relations~\eqref{eq:recurrence_rel_p} 
will satisfy an expression of the above type with appropriately chosen initial conditions.)
Let
\begin{equation}
\F(z)\;=\;\left[\begin{array}{c}F_{\kay}(z)\\ F_{\kay+1}(z)\\ \vdots\\ F_{\kay+\ell-1}(z)\end{array}\right] \qquad \mbox{and} \qquad
\G(z)\;=\;\left[\begin{array}{c}G_{\kay}(z)\\ G_{\kay+1}(z)\\ \vdots\\ G_{\kay+\ell-1}(z)\end{array}\right]
\end{equation}
and define
\begin{equation}
\begin{gathered}
\I(z)\;=\;\left[\begin{array}{c}I_{\kay+2}(z)-I_{\kay}(z)\\ I_{\kay+3}(z)-I_{\kay}(z)\\ \vdots\\ I_{\kay+\ell+1}(z)-I_{\kay}(z)\end{array}\right] \ , \qquad
\J(z)\;=\;\left[\begin{array}{c}J_{\kay+2}(z)-J_{\kay}(z)\\ J_{\kay+3}(z)-J_{\kay}(z)\\ \vdots\\ J_{\kay+\ell+1}(z)-J_{\kay}(z)\end{array}\right] \ , \\[5pt]
\K(z)\;=\;\left[\begin{array}{c}A^{(1)}_{(\kay+\ell)}z^{\kay+1}\\ A^{(2)}_{(\kay+\ell)}z^{\kay+2}\\ \vdots\\ A^{(\ell)}_{(\kay+\ell)}z^{\kay+\ell} \end{array}\right]
\ .
\end{gathered}
\end{equation}
Setting
\begin{equation}
\L(z)\;=\;
\left[\begin{array}{cccc}
L_{11}(z) & L_{12}(z) & \cdots & L_{1\ell}(z)\\ 
L_{21}(z) & L_{22}(z) &        & \vdots \\ 
\vdots    &           & \ddots &\vdots \\ 
L_{\ell1}(z)&\cdots &\cdots & L_{\ell\ell}(z)
\end{array}\right] \ ,
\end{equation}
where
\begin{equation}
L_{mn}(z)
\;=\;
1-A^{(m)}_{(\kay+\ell+n-1)}z^m-B^{(m)}_{(\kay+\ell+n)}z^{m+1} \ ,
\end{equation}
the expressions~\eqref{eq:F_relation} and~\eqref{eq:G_relation} above respectively become
\begin{equation}\label{eq:F+G_relations-mx_form}
\L(z) \F(z)\;=\;\I(z)-p_{\kay} \K(z) \qquad \mbox{and} \qquad
\L(z) \G(z)\;=\;\J(z)-q_{\kay} \K(z) \ .
\end{equation}
We wish to solve the above equations~\eqref{eq:F+G_relations-mx_form} for $\F(z)$ and $\G(z)$. 
Let
\begin{equation}
\M(z)
\;=\;
\left[\begin{array}{cccccc}
1       & 0      & \cdots & \cdots &\cdots & 0\\
-z^{-1} & z^{-1} & 0      & \cdots &\cdots & 0\\
0       & -z^{-2} & z^{-2}& 0      &       & 0\\
\vdots  &         & \ddots& \ddots &\ddots & \vdots\\
\vdots  &         &       & \ddots &\ddots & 0\\
0       & \cdots &        & 0      & -z^{-\ell+1} & z^{-\ell+1} 
\end{array}\right] \ .
\end{equation}
To simplify, apply $\M(z)$ to both sides of~\eqref{eq:F+G_relations-mx_form}.
After recalling (cf. equations~\eqref{eq:AB_recurrence_rel}) that we have the initial conditions 
$B^{(1)}_{(n)}=1$ and $A^{(1)}_{(n)}=a_{n+1}$, 
$a_{\kay+n+\ell}=a_{\kay+n}$ for $1\leq n<\ell+1$, 
and we also have the relations:
\begin{align}
B^{(m)}_{(\kay+n+1)}&=A^{(m-1)}_{(\kay+n+2)}\\
A^{(m)}_{(\kay+n)}&=B^{(m-1)}_{(\kay+n+1)}+A^{(1)}_{(\kay+n)}A^{(m-1)}_{(\kay+n+1)}=A^{(m-2)}_{(\kay+n+2)}+a_{\kay+n+1}A^{(m-1)}_{(\kay+n+1)}
\end{align}
we find that the matrix $\N(z)=\M(z)\L(z)$ can be expressed
in column form as
\begin{equation}\label{eq:O-col_form}
\begin{gathered}
\left[\begin{array}{cccc}
\A_\ell-a_{\kay+1}z \A_1-z^2 \A_2 & \A_1-a_{\kay+2}z \A_2-z^2 \A_3 &\cdots & \A_{\ell-1}-a_{\kay+\ell} z\A_\ell-z^2 \A_1\\
\end{array}\right]\\
\;=\;
\left[\begin{array}{cccc}
\A_\ell^{0}+\A_1^{1}+\A_2^{2} & \A_1^{0}+\A_2^{1}+\A_3^{2} &\cdots & \A_{\ell-1}^{0}+\A_{\ell}^{1}+\A_{1}^{2}\\
\end{array}\right]
\end{gathered}
\end{equation}
where, throughout, addition in lower indices is taken modulo $\ell$ and
\begin{equation}
\A_k^{0}\;=\;
\A_k\;=\;
\left[\begin{array}{c}
1\\ A^{(1)}_{(\kay+n)}\\ A^{(2)}_{(\kay+n)}\\ \vdots\\ A^{(\ell-1)}_{(\kay+n)} 
\end{array}\right] \ ,
\ \qquad
\A^{1}_{n}\;=\;-a_{\kay+n} z \A_{k} \ ,
\ \qquad
\A^{2}_{n}\;=\;-z^2 \A_{n} \ .
\end{equation}
Thus we have an expression for $\M(z)$ applied to the left-hand side of~\eqref{eq:F+G_relations-mx_form}.
A direct calculation gives an expression for the right-hand side.
Namely
\begin{comment}
\begin{equation}
\begin{gathered}
\M(z)\I(z)
\;=\;
p_\kay z^\kay\left[\begin{array}{c}1\\0\\\vdots\\0\end{array}\right]
+z^{\kay+1}\left[\begin{array}{c}p_{\kay+1}\\ p_{\kay+2}\\\vdots\\ p_{\kay+\ell}\end{array}\right]\quad
\M(z)\J(z)
\;=\;
q_\kay z^\kay\left[\begin{array}{c}1\\0\\\vdots\\0\end{array}\right]
+z^{\kay+1}\left[\begin{array}{c}q_{\kay+1}\\ q_{\kay+2}\\\vdots\\ q_{\kay+\ell}\end{array}\right]\\
\M(z)\K(z)
=
z^{\kay+1}
\left[\begin{array}{c}A^{(1)}_{(\kay+\ell)}\\A^{(2)}_{(\kay+\ell)}\\\vdots\\A^{(\ell)}_{(\kay)}\end{array}\right]
-z^\kay
\left[\begin{array}{c}0\\A^{(1)}_{(\kay+\ell)}\\\vdots\\A^{(\ell-1)}_{(\kay)}\end{array}\right]\\
\end{gathered}
\end{equation}
Hence
\end{comment}
\begin{align}
\M(z)\left(\I(z)-p_\kay \K(z)\right)
&\;=\;
p_\kay z^\kay \left[\begin{array}{c}1\\A^{(1)}_{(\kay+\ell)}\\\vdots\\A^{(\ell-1)}_{(\kay+\ell)}\end{array}\right]
+z^{\kay+1} \left[\begin{array}{c}p_{\kay+1}-p_\kay A^{(1)}_{(\kay+\ell)}\\p_{\kay+2}-p_\kay A^{(2)}_{(\kay+\ell)}\\\vdots\\p_{\kay+\ell}-p_\kay A^{(\ell)}_{(\kay+\ell)}\end{array}\right]\label{eq:M(I-pK)}
\end{align}
and
\begin{align}
\M(z)\left(\J(z)-p_\kay \K(z)\right)
&\;=\;
q_\kay z^\kay \left[\begin{array}{c}1\\A^{(1)}_{(\kay+\ell)}\\\vdots\\A^{(\ell-1)}_{(\kay+\ell)}\end{array}\right]
+z^{\kay+1} \left[\begin{array}{c}q_{\kay+1}-q_\kay A^{(1)}_{(\kay+\ell)}\\q_{\kay+2}-q_\kay A^{(2)}_{(\kay+\ell)}\\\vdots\\q_{\kay+\ell}-q_\kay A^{(\ell)}_{(\kay+\ell)}\end{array}\right] \ .\label{eq:M(J-pK)}
\end{align}
Consequently, applying $\M(z)$ to both sides of the equations~\eqref{eq:F+G_relations-mx_form} gives
\begin{align}\label{eq:F+G_relations-mx-form-2}
\N(z)\F(z)\;=\;p_\kay z^\kay \A_\ell+z^{\kay+1}\U \ , \qquad
\N(z)\G(z)\;=\;q_\kay z^\kay \A_\ell+z^{\kay+1}\V \ .
\end{align}
where $\U$ and $\V$ are the right-most vectors in~\eqref{eq:M(I-pK)} and~\eqref{eq:M(J-pK)} respectively.
\begin{remark}\label{rmk:important-indexing}
The above analysis works for {\it any} pre-period $\kay$, not just the minimal pre-period.
However, to achieve a succinct expression for the generating functions, avoiding a case analysis,
we will see below that it is better to choose {\it any} other pre-period. 
We will draw attention to the two places in the proof below where this is necessary.
\end{remark}
With the above remark in mind, a key ingredient is the following proposition.
\begin{proposition}
Provided that $\kay$ is non-zero, 
$\U=p_{\kay-1}\A_1$ and $\V=q_{\kay-1}\A_1$. 
Consequently
the equations~\eqref{eq:F+G_relations-mx-form-2} may be expressed as
\begin{equation}\label{eq:F+G_relations-mx-form-3}
\N(z)\F(z)
\;=\;p_\kay z^\kay\A_\ell+p_{\kay-1}z^{\kay+1}\A_1 \ , \qquad
\N(z)\G(z)
\;=\;q_\kay z^\kay\A_\ell+q_{\kay-1}z^{\kay+1}\A_1 \ .
\end{equation}
\end{proposition}
\begin{proof}
Provided that $\kay$ is non-zero, 
$A^{(m)}_{(\kay)}$ is defined for $m\in\mathbb{N}_0$. 
Moreover, it satisfies $A^{(m)}_{(\kay+\ell)}=A^{(m)}_{(\kay)}$ and 
by~\eqref{eq:AB_expansion} and~\eqref{eq:AB_recurrence_rel}
we have the equalities
\begin{equation}
p_{\kay+m}-p_\kay A^{(m)}_{(\kay+\ell)}
=p_{\kay-1} A^{(m-1)}_{(\kay+1)} \ ,
\qquad
q_{\kay+m}-q_\kay A^{(m)}_{(\kay+\ell)}
=q_{\kay-1} A^{(m-1)}_{(\kay+1)} \ .
\end{equation}
Hence, for $\kay$ non-zero, equations~\eqref{eq:M(I-pK)} and~\eqref{eq:M(J-pK)} become respectively:
\begin{align}
\M(z)\left(\I(z)-p_\kay\K(z)\right)
&=p_\kay z^\kay\A_\ell+p_{\kay-1}z^{\kay+1}\A_1\\
\M(z)\left(\J(z)-q_\kay\K(z)\right)
&=q_\kay z^\kay\A_\ell+q_{\kay-1}z^{\kay+1}\A_1 \ ,
\end{align}
from which the proposition follows immediately.
\end{proof}
Next, define the $\ell\times \ell$ matrix
\begin{equation}
\A=\left[\begin{array}{cccc} \A_1 & \A_2 & \cdots & \A_\ell\end{array}\right] \ .
\end{equation}
We will need the following preliminary result, which follows from a straightforward proof by induction.
\begin{proposition}\label{prop:detA_nonzero}
Let the matrix $\A$ be defined as above.
Then $\det \A=0$ if and only if the sequence 
$a_{\kay+1},a_{\kay+2},\ldots,a_{\kay+\ell}$ 
has period strictly less than $\ell$.
\end{proposition}
%
\begin{comment}
\begin{proof}
Postponed.
\end{proof}
\end{comment}
%
We will also need the following proposition.
However, before proceeding we recommend that the reader first consults Appendix~\ref{sect:matrix_comp}, 
which contains the solution to linear equations with polynomial coefficients of a more general type than those considered below.
\begin{proposition}\label{prop:NE=Ak}
Let $\N(z)$ and $\A_1,\A_2,\ldots,\A_\ell$ be as above.
For each $n\in \mathbb{Z}_\ell$ the equation
\begin{equation}\label{eq:NE=Ak}
\N(z)\E(z)\;=\;\A_n
\end{equation}
has solution $\E_n$ with entries $E_{m,n}$ given by
\begin{equation}\label{eq:E_jk}
E_{m,n}(z)\;=\;\frac{u_{m,n}(z)}{v(z^\ell)}
\end{equation}
where, if $\kappa=\kappa(m,n+1)$ and $\mu=\mu(m,n+1)$ are defined as in Theorem~\ref{thm:inhomog_eq-sol}, then
\begin{equation}\label{eq:E_jk-numerator}
u_{m,n}(z)
\;=\;
\left\{\begin{array}{ll}
z^{(\ell-1)-\kappa}A^{(\ell-\kappa-1)}_{(\kay+n)}+z^{(\ell-1)+\mu}(-1)^\mu A^{(\ell-\mu-1)}_{(\kay+m)} & m\neq n\\
z^{\ell-1}A^{(\ell-1)}_{(\kay+n)} & m=n
\end{array}\right.
\end{equation}
and
\begin{align}\label{eq:alpha_explicit}
v(\xi)
&\;=\;
1
-\left(A^{(\ell-2)}_{(\kay+\ell+1)}+A^{(\ell)}_{(\kay+\ell)}\right)\xi
+(-1)^\ell\xi^2 \ .
\end{align}
\end{proposition}
\begin{proof}
The matrix $\N(z)$ is expressible in the form~\eqref{eq:O-col_form} 
and, by Proposition~\ref{prop:detA_nonzero} above, $\det\A$ is non-zero.
Therefore we may apply Theorem~\ref{thm:inhomog_eq-sol}.
More precisely, for $m=1,2,\ldots,\ell$, if we take 
\begin{equation}\label{eq:gamma-rotn}
\gamma_m^0\;=\;1 \ , \qquad \gamma_m^1\;=\;-a_{\kay+m} \ , \qquad \gamma_m^2\;=\;-1 \ , 
\end{equation}
we take $\C_m=\A_{m-1}$, and
we set $t=n+1$, then Theorem~\ref{thm:inhomog_eq-sol} implies 
that the solution to~\eqref{eq:NE=Ak}, which we denote by $\E_n(z)$, exists and has entries $E_{m,n}(z)$, 
$m=1,2,\ldots,\ell$, given by
\begin{equation}
E_{m,n}(z)\;=\;\frac{u_{m,n}(z)}{v(z^\ell)}
\end{equation}
where $v(z)$ and $u_{m,n}(z)$ are polynomials in the variable $z$.
In fact, $u_{m,n}(z)$ is given by
\begin{equation}\label{eq:u_jk}
u_{m,n}(z)\;=\;\left\{
\begin{array}{ll}
z^{(\ell-1)-\kappa}u_{m,n+1}^{\kappa,0}+z^{(\ell-1)+\mu}u_{m,n+1}^{0,\mu} & m\neq n\\
z^{\ell-1}u_{n,n+1}^{0,0} & m=n
\end{array}
\right.
\end{equation}
where 
$\kappa=\kappa(m,n+1)$ and 
$\mu=\mu(m,n+1)$ 
are given by~\eqref{eq:kappa_def} and~\eqref{eq:mu_def}, 
and the coefficients 
$u_{m,n}^{\kappa,0}$ and 
$u_{m,n}^{0,\mu}$ 
satisfy the recurrence relations
~\eqref{eq:kappa_recurrence_rel} 
and~\eqref{eq:mu_recurrence_rel} respectively.
Observe that the initial values for the first recurrence satisfy
\begin{align}\label{eq:(kappa,0)-init_cond}
u_{n+1,n+1}^{\ell-1,0}
\;=\;A^{(0)}_{(\kay+n)} \ , 
\qquad
u_{n+2,n+1}^{\ell-2,0}
\;=\;A^{(1)}_{(\kay+n)} \ .
\end{align}
Observe also that the recurrence relation~\eqref{eq:kappa_recurrence_rel}, 
after making the substitutions in~\eqref{eq:gamma-rotn} above, 
becomes~\eqref{eq:AB_splitting_rel_2}.
Moreover, the initial conditions~\eqref{eq:(kappa,0)-init_cond} above
agree with the initial conditions given by~\eqref{eq:AB_splitting_rel_2}.
Consequently, for all $m=1,2,\ldots,\ell$, we have
\begin{equation}\label{eq:beta-kappa-rotn}
u_{m,n+1}^{\kappa,0}\;=\;A^{(\ell-\kappa-1)}_{(\kay+n)} \ .
\end{equation}
Next observe that the second recurrence has initial values 
\begin{align}\label{eq:(0,mu)-init_cond}
u_{n-1,n+1}^{0,\ell-1}
\;=\;(-1)^{\ell-1}A^{(0)}_{(\kay+n-1)} \ , 
\qquad 
u_{n-2,n+1}^{0,\ell-2}
\;=\;(-1)^{\ell-2}A^{(1)}_{(\kay+n-2)} \ ,
\end{align}
so the recurrence relation~\eqref{eq:mu_recurrence_rel}
together with these initial conditions~\eqref{eq:(0,mu)-init_cond} agree with the 
recurrence relation and initial conditions~\eqref{eq:AB_splitting_rel_1}.
Thus, for all $m=1,2,\ldots,\ell$, we have
\begin{equation}\label{eq:beta-mu-rotn}
u_{m,n+1}^{0,\mu}\;=\;(-1)^\mu A^{(\ell-\mu-1)}_{(\kay+m)} \ .
\end{equation}
(Note: here we have used the convention stated in Remark~\ref{rmk:important-indexing} above.)
From this we get the expression for the numerator of 
$E_{m,n}$ in terms of $A^{(\ell-\kappa-1)}_{(\kay+n)}$ and $A^{(\ell-\mu-1)}_{(\kay+m)}$ given by~\eqref{eq:E_jk-numerator}.

Now consider the denominator of $E_{m,n}$.
Theorem~\ref{thm:det-general} tells us that $v$ is of the form
\begin{equation}\label{eq:v_in_terms_of_v0v1v2}
v(\xi)\;=\;v_0+\xi v_1+\xi^2v_2
\end{equation}
where the coefficients are given by
\begin{equation}\label{eq:v0v1v2}
\begin{gathered}
v_0
\;=\;\prod_{1\leq s\leq \ell}\gamma_s^0
\;=\;1 \ ,
\qquad\qquad\qquad 
v_2
\;=\;\prod_{1\leq s\leq \ell}\gamma_s^2
\;=\;(-1)^\ell \ , \\[5pt]
v_1
\;=\;\sum_{\tau\in\mathcal{T}^+: \tau\not\equiv 0,2}
(-1)^{\sgn(\upsilon_\tau)}\prod_{1\leq s\leq \ell}\gamma_s^{\tau(s)} \ .
\end{gathered}
\end{equation}
The equalities~\eqref{eq:beta-kappa-rotn} and~\eqref{eq:beta-mu-rotn},
and the hypothesis~\eqref{eq:gamma-rotn} 
together with Proposition~\ref{prop:alpha_in_terms_of_beta}, in the case $s=\ell+1$,
(recalling that addition in lower indices is taken in $\mathbb{Z}_\ell$)
imply that
\begin{align}
v_1
&\;=\;
u_{\ell+1,\ell+1}^{0,1}+(-a_{\kay+\ell+1})u_{\ell+1,\ell+2}^{0,0}+(-1) u_{\ell+1,\ell+3}^{1,0}\\
&\;=\;
-A^{(\ell-2)}_{(\kay+\ell+1)}-a_{\kay+\ell+1}A^{(\ell-1)}_{(\kay+\ell+1)}-A^{(\ell-2)}_{(\kay+\ell+2)} \ .
\end{align}
Applying the recurrence relation~\eqref{eq:AB_splitting_rel_1}, 
we therefore find that
\begin{equation}\label{eq:v1_in_terms_of_A}
v_1\;=\;-A^{(\ell-2)}_{(\kay+\ell+1)}-A^{(\ell)}_{(\kay+\ell)} \ ,
\end{equation}
which proves~\eqref{eq:alpha_explicit}. Hence the proposition is shown.
\end{proof}
We are now in a position to give a proof of the main theorem (Theorem~\ref{thm:main-gen_fn_rat2}). 
\begin{proof}[Proof of Theorem~\ref{thm:main-gen_fn_rat2}]
Assume that $\kay\geq 2$.
By the analysis above, it suffices to solve the equations~\eqref{eq:F+G_relations-mx-form-3} for $\F(z)$ and $\G(z)$.
We will only consider the equation for $\F(z)$ as the argument for $\G(z)$ is totally analogous.

Applying Proposition~\ref{prop:NE=Ak} in the case $n=1$ and $n=\ell$ to the equation~\eqref{eq:F+G_relations-mx-form-3}
and then applying linearity gives
\begin{equation}
\F(z)
\;=\;
p_\kay z^\kay \E_\ell(z)+p_{\kay-1}z^{\kay+1}\E_1(z) \ .
\end{equation}
More specifically, for $m=1,2,\ldots,\ell$, using the equalities
~\eqref{eq:E_jk},
~\eqref{eq:E_jk-numerator}
and~\eqref{eq:alpha_explicit}
\begin{equation}
F_{\kay+m-1}(z)
\;=\;p_{\kay} z^{\kay} E_{m,\ell}(z)+p_{\kay-1}z^{\kay+1} E_{m,1}(z)
\;=\;\frac{z^\kay u_m(z)}{v(z^\ell)}
\end{equation}
where, after rearranging and collecting like terms, we have
\begin{align}
&
u_m(z)\;=\;\\
&
\left\{\begin{array}{ll}
z^{0}p_{\kay}A^{(0)}_{(\kay+\ell)}
+ z^{\ell}\left( (-1)^1p_{\kay}A^{(\ell-2)}_{(\kay+1)}+p_{\kay-1}A^{(\ell-1)}_{(\kay+1)} \right)
& m=1\\[10pt]
z^{m-1}\left(p_{\kay}A^{(m-1)}_{(\kay+\ell)}+p_{\kay-1}A^{(m-2)}_{(\kay+1)}\right)
&\\
 \ \ + z^{\ell+m-1}\left( (-1)^mp_{\kay}A^{(\ell-m-1)}_{(\kay+m)}+(-1)^{m-1}p_{\kay-1}A^{(\ell-m)}_{(\kay+m)}\right) 
& 2\leq m\leq\ell-1\\[10pt]
z^{\ell-1}\left(p_{\kay}A^{(\ell-1)}_{(\kay+\ell)}+p_{\kay-1}A^{(\ell-2)}_{(\kay+1)}\right)
+ z^{2\ell-1}(-1)^{\ell-1}p_{\kay-1}A^{(0)}_{(\kay+\ell)}
& m=\ell
\end{array}\right.\notag
\end{align}
and
\begin{equation}
v(\xi)\;=\;1-\left(A^{(\ell-2)}_{(\kay+\ell+1)}+A^{(\ell)}_{(\kay+\ell)}\right)\xi+(-1)^\ell\xi^2 \ .
\end{equation}
But by equations
~\eqref{eq:pA_splitting_rel_1} and~\eqref{eq:pA_splitting_rel_2} 
of Proposition~\ref{prop:pA&qA_splitting_rel}, 
together with the periodicity of $A^{(m)}_{(n)}$ 
in the lower index and the hypothesis that $\kay$ 
is non-zero and greater than or equal to $\ell$, we obtain
\begin{align}
p_{\kay}A^{(m-1)}_{(\kay+\ell)}+p_{\kay-1}A^{(m-2)}_{(\kay+1)}
&\;=\;
p_{\kay+m-1}\\ 
(-1)^m p_{\kay}A^{(\ell-m-1)}_{(\kay+m)}+(-1)^{m-1} p_{\kay-1}A^{(\ell-m)}_{(\kay+m)}
&\;=\;
(-1)^{\ell+1} p_{\kay-\ell+m-1} 
\end{align}
and therefore, for $1\leq m\leq \ell$,
\begin{align}
u_m(z)
\;=\;
z^{m-1}p_{\kay+m-1}
+z^{\ell+m-1}(-1)^{\ell+1}p_{\kay-\ell+m-1} \ .
\end{align}
(Observe that the exceptional expressions for 
$u_0$ and $u_\ell$ coincide with the general expression for $u_m$ when substituting $m=1$ and $m=\ell$ respectively.) 
By applying Proposition~\ref{prop:trace_AB_pq} we get the equivalent expression
\begin{align}
v_1
&\;=\;
-\left(A^{(\ell-2)}_{(\kay+1)}+A^{(\ell)}_{(\kay)}\right)\\
&\;=\;
-(-1)^{\kay}\left(
q_{\kay}p_{\kay+\ell-1}-p_{\kay}q_{\kay+\ell-1}-q_{\kay-1}p_{\kay+\ell}+p_{\kay-1}q_{\kay+\ell}
\right) \ .
\end{align}
Consequently, equation~\eqref{eq:Fn} holds. 
By exactly the same argument,
replacing where appropriate the $p_n$ with $q_n$, 
gives equation~\eqref{eq:Gn}. 
Thus the theorem is shown.
\end{proof}
\begin{eg}
Let $\theta=\frac{-1+\sqrt{5}}{2}$.
Then $\theta$ has simple continued fraction expansion $[1,1,1,\ldots]$.
Then $\ell=1$ and we take $\kay=2$.
Observe that
$\delta_\theta=1$.
Theorem~\ref{thm:main-gen_fn_rat2} then gives, using $p_0=0$, $p_1=1$ and $p_2=p+1+p_0$,
the following expression for $F(z)$:
\begin{align}
F(z)
\;=\;p_0+p_1 z+F_2(z)
&\;=\;p_0+p_1 z+\frac{p_2 z^2+(-1)^2 p_1 z^3}{1-z-z^2}\\
&\;=\;\frac{p_0+(p_1-p_0)z+(p_2-p_1-p_0)z^2}{1-z-z^2}\\
&\;=\;\frac{z}{1-z-z^2} \ .
\end{align}
Exactly the same argument, {\it mutatis mutandis}, using $q_0=1$, $q_1=1$ and $q_2=q_1+q_0$, gives the following expression for $G(z)$:
\begin{equation}
G(z)
\;=\;q_0+q_1z+G_2(z)
\;=\;\frac{q_0+(q_1-q_0)z+(q_2-q_1-q_1)z^2}{1-z-z^2}
\;=\;\frac{1}{1-z-z^2} \ .
\end{equation}
\end{eg}
Since the matrix $\A$ is nonsingular, 
it follows that a similar analysis can be made for an arbitrary right-hand side of equations~\eqref{eq:F+G_relations-mx_form}.
In particular, given {\it any} initial conditions to the recurrence 
relations~\eqref{eq:recurrence_rel_p} and~\eqref{eq:recurrence_rel_q}.
\begin{corollary}
Take an arbitrary pre-periodic sequence $a_j$ of positive integers of pre-period $\ell$.
Let $r_n$ be an arbitrary sequence satisfying the recurrence relations given in~\eqref{eq:recurrence_rel_p} and~\eqref{eq:recurrence_rel_q}.
(For any choice of initial conditions.)
Then the corresponding generating function is rational.
\end{corollary}
\begin{corollary}\label{cor:alpha_vs_chi}
Let $\theta\in\mathbb{R}\setminus\mathbb{Q}$ be a quadratic irrational.
Let $M_1$ denote the element of $\mathrm{PSL}(2,\mathbb{Z})$ corresponding to the periodic part of the continued fraction expansion of $\theta$.
Denote the characteristic polynomial by $\chi$.
The denominator of the corresponding generating functions $F$ and $G$ has the form $v(z^\ell)$,
where 
\begin{equation}\label{eq:denom_vs_char-poly}
v(\xi)
\;=\;
\xi^2\chi\left(\frac{1}{\xi}\right) \ .
\end{equation}
In particular, 
(i) $\discr_v=\discr_\chi$; 
(ii) the zero of $v$ of smallest modulus is the reciprocal of the zero of $\chi$ of largest modulus, 
{\it i.e.}, the reciprocal of the largest eigenvalue of $M_1$.
\end{corollary}
\begin{proof}
Recall that $M_1$ is given by equation~\eqref{eq:a1a2...al}. 
Observe that $M_1$ is a product of $\ell$ matrices of determinant $-1$.
Thus $\det M_1=(-1)^\ell$.
Next observe that by equations
~\eqref{eq:v1_in_terms_of_A},
~\eqref{eq:v0v1v2}
and~\eqref{eq:v_in_terms_of_v0v1v2}

\begin{align}
v(\xi)
&=1-\left(A^{(\ell-2)}_{(\kay+1)}+A^{(\ell)}_{(\kay)}\right)\xi+(-1)^\ell\xi^2\\
&=1-\mathrm{tr}(M_1)\xi+\mathrm{det}(M_1)\xi^2\\
&=\xi^2\chi\left(\frac{1}{\xi}\right) \ .
\end{align}
This shows the first part.
Since the discriminant of an arbitrary quadratic polynomial $\chi$ 
is invariant under the involution 
$\chi(\xi)\mapsto \xi^2 \chi(1/\xi)$, 
(i) holds. 
Finally, (ii) follows trivially from~\eqref{eq:denom_vs_char-poly}.

\begin{comment}
From the preceding proof we $v$ is given explicitly 
by equation~\eqref{eq:alpha_explicit}.
A straighforward calculations shows therefore that
the discriminant is
\begin{equation}\label{eq:discr_alpha=discr_chi}
\discr_v=\left(A^{(\ell-2)}_{(1)}+A^{(\ell)}_{(0)}\right)^2-4(-1)^\ell
\end{equation}
Since the minimal polynomial $\chi$ of $\theta$ 
is given by~\eqref{eq:min_poly} we find also that
\begin{equation}\label{eq:discr_chi=discr_alpha}
\discr_\chi
=\left(A^{(\ell)}_{(0)}-A^{(\ell-2)}_{(1)}\right)^2+4A^{(\ell-1)}_{(0)}A^{(\ell-1)}_{(1)}
=\left(A^{(\ell)}_{(0)}+A^{(\ell-2)}_{(1)}\right)^2-4
\left(A^{(\ell-2)}_{(1)}A^{(\ell)}_{(0)}-A^{(\ell-1)}_{(0)}A^{(\ell-1)}_{(1)}\right)
\end{equation}
but since the matrix in the right-hand side of 
equation~\eqref{eq:a1a2...al} is a product of 
$\ell$ matrices of determinant $-1$, we have
\begin{equation}\label{eq:det_a1a2...al}
A^{(\ell-2)}_{(1)}A^{(\ell)}_{(0)}-A^{(\ell-1)}_{(0)}A^{(\ell-1)}_{(1)}=(-1)^\ell
\end{equation}
Combining~\eqref{eq:discr_alpha=discr_chi},
~\eqref{eq:discr_chi=discr_alpha} 
and~\eqref{eq:det_a1a2...al} now gives the result.
\end{comment}
\end{proof}
%
\begin{comment}
\begin{remark}
Let $\chi$ and $\hat\chi$ be quadratic polynomials with integer coefficients.
If $\discr_{\chi}=\discr_{\hat\chi}$, then there exists an affine map 
with rational coefficients sending the zeroes of $\chi$ to the zeroes of $\hat\chi$.
In other words, the zero sets are rationally related.
\end{remark}
\end{comment}
%
\begin{proof}[Proof of Theorem~\ref{thm:Levy_const-vs-spec_radius}]
By the Theorem of Jager and Liardet~\cite{JagerLiardet1988}, 
the L\'evy constant 
$\beta(\theta)=\lim_{n\to\infty} \frac{1}{n}\log q_n$ 
exists.
By the Cauchy-Hadamard formula, 
if $\rho$ denotes the radius of convergence of the generating function $G$, 
then
$1/\rho=\lim_{n\to\infty}q_n^{1/n}$.
Therefore, 
by continuity of the logarithm,
$\beta(\theta)=-\log\rho$.

Recall that, 
analogously to equation~\eqref{eq:F_splitting-k=0}, 
$G(z)=J_\kay(z)+\sum_{\kay\leq n\leq\kay+\ell} G_n(z)$, 
where $G_n$ is defined by~\eqref{def:Fn+Gn}.
By Theorem~\ref{thm:main-gen_fn_rat2}, each $G_n$ may be expressed in the form
\begin{equation}
G_n(z)=u_n(z)/v(z^\ell)
\end{equation} 
where 
\begin{equation}
u_n(z)=z^n q_n+z^{n+\ell}(-1)^{\ell+1}q_{n-\ell} \quad \mbox{and} \quad
v(\xi)=1-(-1)^\kay\delta \xi+(-1)^{\ell}\xi^2 \ ,
\end{equation}
for $\kay\leq n<\kay+\ell$.
(Observe the indexing is different from the proof of Theorem~\ref{thm:main-gen_fn_rat2}.)
Consequently,
\begin{equation}
G(z)=I_\kay(z)+u(z)/v(z^\ell) \ ,
\end{equation} 
where $u(z)=\sum_{\kay\leq n\leq\kay+\ell} u_n(z)$.
The poles of $G$ must occur at solutions to $v(z^\ell)=0$.
Let $v_\pm$ denote the zeroes of $v$, where $\pm$ denotes the sign used in the quadratic formula.
Thus
\begin{equation}
v_\pm
\;=\;
\frac{\tau\pm\sqrt{\tau^2-4(-1)^\ell}}{2(-1)^\ell}
\end{equation}
where $\tau=\mathrm{tr}M_1=A^{(\ell-2)}_{\kay+1}+A^{(\ell)}_{\kay}$.
Observe that both $v_-$ and $v_+$ are real. 
Also observe that the zero $v_\mathrm{min}$ with the smallest modulus 
($v_-$ in the case when $\ell$ is even, and $v_+$ in the case when $\ell$ is odd)
is always positive and, in fact, lies in the interval $(0,1)$.
Consequently $v_\mathrm{min}^{1/\ell}$ is a zero of $v(z^\ell)$, it is also positive real number, and moreover lies in the interval $(0,1)$.

\begin{claim}
The polynomial $u(z)$ has no zeroes in $(0,1)$.
\end{claim}
\vspace{5pt}

\noindent
{\it Proof of Claim:}
Rearranging the expression for $u(z)$ given above, we have
\begin{equation}
u(z)
\;=\;\sum_{\kay\leq n\leq \kay+\ell} u_n(z)
\;=\;\sum_{\kay\leq n\leq \kay+\ell} z^{n}\left(q_{n} +z^\ell(-1)^{\ell+1} q_{n-\ell}\right) \ .
\end{equation}
Observe that $q_n>q_{n-\ell}$ for all $n>\ell$. 
Thus, for all $n>\ell$ and all $z\in (0,1)$, we have the inequality
$q_{n}>|q_{n-\ell}z^\ell|$ 
On the interval $(0,1)$ the polynomial $u$ can be expressed as a sum of positive terms, and the claim follows.
/\!/

\vspace{5pt}

\noindent
Consequently, $v_\mathrm{min}^{1/\ell}$ is a pole of the generating function $G(z)$.
Any other pole of $G(z)$ must lie on the union of circles
\begin{equation}
\left\{|z|=|v_-|^{1/\ell}\right\} 
\cup 
\left\{|z|=|v_+|^{1/\ell}\right\} \ .
\end{equation}
Thus $\rho=|v_\mathrm{min}|^{1/\ell}$.
However, by Corollary~\ref{cor:alpha_vs_chi} above, 
it follows that $1/v_-$ and $1/v_+$ are zeroes of the characteristic polynomial $\chi$.
Since $v_\mathrm{min}$ is the zero of $v(\xi)$ of minimal modulus
it follows that 
$1/v_\mathrm{min}$ is the zero of $\chi(\xi)$ of maximal modulus, {\it i.e.}, 
an eigenvalue of maximal modulus of the matrix $M_\theta$. 
The result follows. 
\end{proof}
%
%

\appendix 
\section{Matrix Computations.}\label{sect:matrix_comp}
\subsection{The Setup.}
Let $\C_1,\C_2,\ldots,\C_\ell$ be column vectors in $\mathbb{C}^\ell$
and define the $\ell\times\ell$ matrix
\begin{equation}
\C
\;=\;
\left[\begin{array}{cccc} \C_1&\C_2&\cdots&\C_\ell\end{array}\right] \ .
\end{equation}
For $1\leq r,s\leq \ell$, denote by
\begin{itemize}
\item 
$\C^{r}$ the matrix $\C$ with the $r$th row removed;
\item 
$\C^{rs}$ the matrix $\C$ with the $r$th row and $s$th column removed\footnote{{\it i.e.}, the $(r,s)$th matrix minor of $\C$.};
\item
$\C_{s}$ the $s$th column of the matrix $\C$;
\item 
$\C_{rs}$ the $r$th row of the $s$th column\footnote{{\it i.e.}, the $(r,s)$th entry of $\C$.} of the matrix $\C$;
\item
$\C_s^r$ the column $\C_s$ with the $r$th row removed.
\end{itemize}
%
For each $s=1,2,\ldots,\ell$, and $p=0,1$ or $2$,
take a polynomial $c_s^p(z)\in\mathbb{C}[z]$ 
and define
\begin{equation}\label{eq:C_mp_cols}
\C_{s;p}(z)
\;=\;
c_s^p(z)\C_{s+p} \ , 
\qquad 
\forall s=1,2,\ldots,\ell , \qquad \forall p=0,1,2 \ .
\end{equation}
(Thus, for instance, 
$\C_{s;p}^r(z)=c_s^p(z)\C_{s+p}^r$, 
where addition in the lower index is taken mod $\ell$.)
Define the column vectors
\begin{equation}\label{eq:O_cols}
\N_s(z)
\;=\;
\C_{s;0}(z)+\C_{s;1}(z)+\C_{s;2}(z) \ , 
\qquad 
\forall s=1,2,\ldots,\ell \ ,
\end{equation}
and the matrix
\begin{equation}
\N(z)
\;=\;
\left[\begin{array}{cccc}\N_1(z) & \N_2(z) & \cdots & \N_\ell(z)\end{array}\right] \ .
\end{equation}
For $1\leq r,s\leq \ell$, let 
\begin{equation}
\N^{rs}(z)
\;=\;
\left[\begin{array}{cccccc}\N_1^r(z)&\N_2^r(z)&\cdots & \widehat{\N_s^r(z)}&\cdots&\N_\ell^r(z)\end{array}\right] \ ,
\end{equation}
where the hat denotes that the column is omitted.
(Thus $\N^{rs}(z)$ denotes the $(r,s)$th matrix minor of $\N(z)$, 
agreeing with the notation already introduced above.)

\subsection{Computation of a Determinant.}\label{sect:determinant}
%
\begin{theorem}\label{thm:det-general}
If
\begin{equation}\label{eq:monomial_hyp}
c^p_s(z)\;=\;\gamma^p_s z^p 
\quad \mbox{for some} \quad 
\gamma^p_s\in\mathbb{C} \ , 
\qquad
\forall s=1,2,\ldots,\ell \ , \ \forall p=0,1,2 \ ,
\end{equation}
then
\begin{equation}
\det\N(z)=\det\C\times v(z^\ell)
\end{equation}
where $v$ is a quadratic polynomial
\begin{equation}
v(\xi)\;=\;v_0+v_1\xi+v_2\xi^2
\end{equation}
with coefficients given by
\begin{itemize}
\item
$v_0=\prod_{1\leq k\leq\ell} \gamma_k^0$;
\item
$v_1=v_1(\gamma_1^0,\ldots)$ 
is a polynomial in 
$\gamma_1^0,\ldots$, 
which is invariant under cyclic permutations in the lower index, but not a symmetric polynomial;
\item
$v_2=\prod_{1\leq k\leq\ell} \gamma_k^2$.
\end{itemize}
\end{theorem}
To prove Theorem~\ref{thm:det-general}, 
we need to set up the following notation and 
terminology. 
(This will be followed by some auxiliary 
propositions before we start the proof.)
We will identify the index set 
$\{1,2,\ldots,\ell\}$ 
for the collection of columns 
$\C_1,\C_2,\ldots,\C_\ell$
with the cyclic group $\mathbb{Z}_\ell=\mathbb{Z}/\ell\mathbb{Z}$.
Thus the index set $\{1,2,\ldots,\ell\}$ 
becomes endowed with addition modulo $\ell$ 
and the cyclic order in a natural manner.
We observe, importantly, that the notion of a closed (oriented) interval makes sense in this setting.
Let 
\begin{equation}
\mathcal{T}
\;=\;
\bigl\{\tau\colon \mathbb{Z}_\ell\to \{0,1,2\}\bigr\} \ .
\end{equation}
We will also represent $\tau$ in $\mathcal{T}$ by the corresponding string,
over the alphabet $\{0,1,2\}$, given by
\begin{equation}\label{eq:tau_string}
\tau(1)\tau(2)\cdots\tau(\ell) \ .
\end{equation}
As usual, we denote by $0^r$ a (sub)string of $00\cdots 0$ of length $r$ and define $1^r$ and $2^r$ similarly.  
\begin{remark}
The reason for introducing this notation is the following.
Since the determinant is multilinear, equation~\eqref{eq:O_cols} implies that
\begin{equation}\label{eq:detO-multilinear}
\det \N(z)
\;=\;
\sum_{\tau\in \mathcal{T}}
\det \left[\begin{array}{cccc}\C_{1;\tau(1)}& \C_{2;\tau(2)}& \cdots& \C_{\ell;\tau(\ell)}\end{array}\right] \ .
\end{equation}
We need to compute each summand. 
For each $\tau\in\mathcal{T}$, multilinearity of the determinant again, together with equation~\eqref{eq:C_mp_cols}, 
implies that
\begin{equation}\label{eq:det_monomial}
\begin{gathered}
\det \left[\begin{array}{cccc}\C_{1;\tau(1)}& \C_{2;\tau(2)}& \cdots& \C_{\ell;\tau(\ell)}\end{array}\right]\\
\;=\;
c(\tau) 
\det\left[\begin{array}{cccc}\C_{1+\tau(1)}& \C_{2+\tau(2)}& \cdots& \C_{\ell+\tau(\ell)}\end{array}\right]
\end{gathered}
\end{equation}
where
\begin{equation}
c(\tau)\;=\;\prod_{1\leq k\leq \ell} c_k^{\tau(k)} \ .
\end{equation}
Thus we wish to determine 
(i) when the determinant on the right-hand side of~\eqref{eq:det_monomial} is zero,
(ii) if the determinant on the right-hand side of~\eqref{eq:det_monomial} is non-zero, calculate $c(\tau)$.
\end{remark}
For $\tau\in\mathcal{T}$, define $\upsilon_\tau\colon \mathbb{Z}_\ell\to\mathbb{Z}_\ell$ by
\begin{equation}
\upsilon_\tau(k) \;=\; k+\tau(k) \ .
\end{equation}
This gives a one-to-one correspondence between $\mathcal{T}$ and the set
\begin{equation}
\mathcal{U}
\;=\;
\bigl\{\upsilon\colon\mathbb{Z}_\ell\to\mathbb{Z}_\ell \ \big| \ 
k\leq \upsilon(k)\leq k+2 \mod\ell, \  \forall k\in\mathbb{Z}_\ell\bigr\} \ .
\end{equation} 
Given $\upsilon\in\mathcal{U}$, let $\tau_\upsilon$ denote the corresponding element of $\mathcal{T}$.
We say that $\tau\in \mathcal{T}$ is {\it decreasing} at 
$k\in \mathbb{Z}_\ell$ 
if 
\begin{equation}\label{eq:tau_decreasing_at_k}
\tau(k+1)\;=\;\tau(k)-1 \qquad \mbox{or} \qquad
\tau(k+2)\;=\;\tau(k)-2 \ .
\end{equation} 
Naturally, if $\tau$ is not decreasing at any point we say it is {\it non-decreasing}.
We say that $\upsilon\in\mathcal{U}$ is {\it non-decreasing} if the corresponding $\tau_\upsilon$ is non-decreasing.
Let
\begin{align}
\mathcal{T}^+&\;=\;\left\{\tau\in\mathcal{T} : \tau \ \mbox{is non-decreasing} \ \right\}
\end{align}
and
\begin{align}
\mathcal{U}^+&\;=\;\left\{\upsilon\in\mathcal{U} : \upsilon \ \mbox{is non-decreasing} \ \right\} \ .
\end{align}
\begin{proposition}\label{prop:U-permutation}
For $\upsilon\in\mathcal{U}$, $\upsilon\in\mathcal{S}_\ell$ 
only if $\upsilon\in \mathcal{U}^+$.
\end{proposition}
\begin{proof}
Take $k\in\mathbb{Z}_\ell$.
If 
$\tau(k+1)=\tau(k)-1$ then $\upsilon(k+1)=\upsilon(k)$.
Similarly, if
$\tau(k+2)=\tau(k)-2$ then $\upsilon(k+2)=\upsilon(k)$.
In either case $\upsilon$ is not injective.
Thus decreasing $\upsilon$ cannot be permutations.

Conversely, assume that $\upsilon\in\mathcal{U}$ is not a permutation.
Thus there exist distinct $k_1,k_2\in\mathbb{Z}_\ell$ such that $\upsilon(k_1)=\upsilon(k_2)$.
Since $k\leq \upsilon(k)\leq k+2$ for all $k\in\mathbb{Z}_\ell$, this implies that $|k_1-k_2|\leq 2$.
A case analysis now finishes the proof. 
\end{proof}
\begin{corollary}\label{cor:decr+non-decr}
Let $\tau\in\mathcal{T}$.
If $\tau\notin\mathcal{T}^+$ then
\begin{equation}
\det 
\left[\begin{array}{cccc}\C_{1;\tau(1)}& \C_{2;\tau(2)}& \cdots& \C_{\ell;\tau(\ell)}\end{array}\right]
\;=\;0 \ .
\end{equation}
Otherwise $\tau\in\mathcal{T}^+$ and
\begin{align}
\det 
\left[\begin{array}{cccc}\C_{1;\tau(1)}& \C_{2;\tau(2)}& \cdots& \C_{\ell;\tau(\ell)}\end{array}\right]
\;=\;(-1)^{\sgn(\upsilon_\tau)}
c(\tau)
\det \C \ ,
\end{align}
where 
$\sgn(\upsilon_\tau)$ denotes parity of the number of adjacent transpositions of the permutation $\upsilon_\tau$.
\end{corollary}
\begin{lemma}\label{lem:key}
Let $[s,t]$ be a closed (oriented) interval in $\mathbb{Z}_\ell$.
If $\tau$ is non-decreasing on $[s,t]$ then $\tau$ has the form
\begin{equation}
\tau(s)\tau(s+1)\cdots\tau(t)
\;=\;
0^\kappa1^{\lambda_1}b 1^{\lambda_2}b\cdots 1^{\lambda_{r-1}}b1^{\lambda_r}2^\mu \ ,
\end{equation}
where $\kappa$, $\lambda_1,\lambda_2,\ldots,\lambda_r$, and $\mu$ are non-negative integers and 
$b$ denotes the block $b$ of length two given by
\begin{equation}
b\;=\;20 \ .
\end{equation}
\end{lemma}
\begin{proof}
The key observation, together with its dual, is the following:
\begin{itemize}
\item
in the string 
$\tau(s)\tau(s+1)\cdots\tau(t)$, 
two consective $2$'s cannot be succeeded by either $0$ or $1$;
\item
in the string
$\tau(s)\tau(s+1)\cdots\tau(t)$, 
two consective $0$'s cannot be preceded by either $2$ or $1$.
\end{itemize}
The reason for the first being that if $\tau(r)=\tau(r+1)=2$ then 
$\tau(r+2)=1$ implies $\tau$ is decreasing at $r+1$, while 
$\tau(r+2)=0$ implies $\tau$ is decreasing at $r$.
The argument in the dual case is analogous.

As a consequence, by the non-decreasing property,
\begin{itemize}
\item 
any substring of $0$'s of length greater than one must be contained in a substring of $0$'s attached to the left boundary of $[s,t]$;
\item
any substring of $2$'s of length greater than one must be contained in a substring of $2$'s attached to the right boundary of $[s,t]$.
\end{itemize}
Elsewhere in the string, $2$ must be followed by $0$ and $0$ must be preceded by $2$.
(We now see the importance of the block $b$.)
\end{proof}

\begin{proposition}\label{prop:sums_of_non-decr_tau}
For any $\tau\in \mathcal{T}^+$, $\sum_k\tau(k)=\eta\ell$ for some $\eta=\eta(\tau)\in\{0,1,2\}$.
\end{proposition}
\begin{proof}
First, observe that there is 
a unique $\tau$ so that $\sum_k\tau(k)=0$, namely $\tau(k)=0$ for all $k$.
Likewise, there is 
a unique $\tau$ so that $\sum_k\tau(k)=2\ell$, namely $\tau(k)=2$ for all $k$.
Thus we just need to show that in all other cases $\sum_k\tau(k)=\ell$.
However, in all other cases, by Lemma~\ref{lem:key} above, as $\tau$ is non-decreasing (and thus non-decreasing on all subintervals) either
\begin{equation}
\tau(1)\tau(2)\cdots\tau(\ell)
\;=\;
1^{\lambda_1}b1^{\lambda_2}b\cdots 1^{\lambda_{r-1}}b1^{\lambda_r}
\end{equation}
or
\begin{equation}
\tau(1)\tau(2)\cdots\tau(\ell)
\;=\;
01^{\lambda_1}b1^{\lambda_2}b\cdots 1^{\lambda_{r-1}}b1^{\lambda_r}2
\end{equation}
for some non-negative integers $\lambda_1,\lambda_2,\ldots,\lambda_r$.
It follows trivially in either of these cases that $\sum_k \tau(k)=\ell$.
\end{proof}
\begin{proof}[Proof of Theorem~\ref{thm:det-general}]
If hypothesis~\eqref{eq:monomial_hyp} 
is satisfied then
\begin{align}
c(\tau)
\;=\;
\gamma(\tau)\times z^{\sum_k \tau(k)}
\end{align}
where
\begin{equation}
\gamma(\tau)
\;=\;
\prod_{1\leq k\leq \ell}\gamma_k^{\tau(k)}
 \ .
\end{equation}
From equation~\eqref{eq:detO-multilinear},
equation~\eqref{eq:det_monomial}, 
the preceding Corollary~\ref{cor:decr+non-decr}, 
we also get
\begin{align}
\det\N(z)
&\;=\;
\sum_{\tau\in \mathcal{T}^+} \det\left[\begin{array}{cccc}\C_{1;\tau(1)}& \C_{2;\tau(2)}& \cdots& \C_{\ell;\tau(\ell)}\end{array}\right]\\
&\;=\;
\sum_{\tau\in \mathcal{T}^+} 
\gamma(\tau)
z^{\sum_k \tau(k)}
\det\left[\begin{array}{cccc}\C_{1+\tau(1)}& \C_{2+\tau(2)}& \cdots& \C_{\ell+\tau(\ell)}\end{array}\right]\\
&\;=\;
\sum_{\tau\in \mathcal{T}^+}
(-1)^{\sgn(\tau)}
\gamma(\tau) 
z^{\sum_k \tau(k)}
\det\left[\begin{array}{cccc}\C_{1}& \C_{2}& \cdots& \C_{\ell}\end{array}\right]\\
&\;=\;
\det\C \times v(z^\ell)
\end{align}
where, 
by Proposition~\ref{prop:sums_of_non-decr_tau}, 
$v$ is the polynomial in $z$ given by
\begin{equation}
v(\xi)
\;=\;
\sum_{\tau\in \mathcal{T}^+}(-1)^{\sgn(\tau)}\gamma(\tau) \xi^{\eta(\tau)} \ .
\end{equation}
Moreover, Proposition~\ref{prop:sums_of_non-decr_tau} shows that $v$ is quadratic in $\xi$.
Observe that the only $\tau$ where $\eta\neq 1$ are $\tau= 0^\ell$ and $\tau= 2^\ell$.
Therefore $v$ has the following form:
\begin{equation}
v(\xi)
\;=\;
\xi^0\gamma(0^\ell) 
+\xi^1\sum_{\tau\in\mathcal{T}^+:\tau\not\equiv 0,2}(-1)^{\sgn(\upsilon_{\tau})}\gamma(\tau) 
+\xi^2\gamma(2^\ell) \ .
\end{equation} 
\end{proof}
%
\subsection{Solution to an inhomogeneous linear equation}\label{sect:Cramers_rule}
\begin{theorem}\label{thm:inhomog_eq-sol}
Fix $t\in\mathbb{Z}_\ell$ and let $\C$ and $\N$ be as in the preceding section.
Suppose that
\begin{equation}
c_s^p(z)\;=\;\gamma_s^p z^p \quad \mbox{for some} \quad 
\gamma_s^p\in\mathbb{C} \ , \qquad \forall s=1,2,\ldots,\ell \ , \ \forall p=0,1,2 \ ,
\end{equation}
and $\mathrm{det}\C\neq 0$. 
Then the solution $\E(z)$ to the inhomogeneous linear equation
\begin{equation}\label{eq:inhomog}
\N(z) \E(z)\;=\;\C_t
\end{equation}
exists and has entries $E_s(z)$ given by
\begin{equation}
E_s(z)\;=\;\frac{u_s(z)}{v(z^\ell)}
\end{equation}
where 
\begin{enumerate}
\item
$v$ denotes the quadratic polynomial from Theorem~\ref{thm:det-general};
\item
if we define
\begin{equation}\label{eq:kappa_def}
\kappa(s,t)
\;=\;
\left\{\begin{array}{ll}
\ell-1 & s=t\\
0 & s=t-1\\
1 & s=t-2\\
\card [s+1,t-1] & \mbox{otherwise}
\end{array}\right.
\end{equation}
and
\begin{equation}\label{eq:mu_def}
\mu(s,t)
\;=\;
\left\{\begin{array}{ll}
1 & s=t\\
0 & s=t-1\\
\ell-1 & s=t-2\\
\card [t-1,s-1] & \mbox{otherwise}
\end{array}\right.
\end{equation}
then $u_s$ is the polynomial
\begin{equation}
u_s(z)
\;=\;
\left\{
\begin{array}{ll}
z^{(\ell-1)-\kappa}u_{s,t}^{\kappa,0}+z^{(\ell-1)+\mu}u_{s,t}^{0,\mu} & s\neq t-1\\
z^{\ell-1}u_{t-1,t}^{0,0} & s=t-1
\end{array}
\right.
\end{equation}
where $u_{s,t}^{\kappa,0}$ and $u_{s,t}^{0,\mu}$ satisfy the recurrence relations
\begin{equation}
\begin{gathered}
\gamma_{s}^{0}u_{s,t}^{\kappa,0}+\gamma_{s-1}^{1}u_{s-1, t}^{\kappa+1, 0}+\gamma_{s-2}^{2}u_{s-2, t}^{\kappa+2, 0}
\;=\;0 \ , \\[10pt]
u_{t,t}^{\ell-1,0}
\;=\;
\prod_{t+1\leq k\leq t-1}\gamma_k^0 \ ,
\qquad\qquad
u_{t+1,t}^{\ell-2,0}
\;=\;
-\gamma_{t}^1\prod_{t+2\leq k\leq t-1}\gamma_k^0 \ ,
\end{gathered}
\label{eq:kappa_recurrence_rel}
\end{equation}
and
\begin{equation}
\begin{gathered}
\gamma_{s}^{2}u_{s, t}^{0,\mu}+\gamma_{s+1}^{1}u_{s+1, t}^{0, \mu+1}+\gamma_{s+2}^{0}u_{s+2, t}^{0, \mu+2}
\;=\;0 \ , \\[10pt]
u_{t-2,t}^{0,\ell-1}
\;=\;
\prod_{t-1\leq k\leq t-3}\gamma_k^2 \ ,
\qquad \qquad
u_{t-3,t}^{0,\ell-2}
\;=\;
-\gamma_{t-2}^{1}\prod_{t-1\leq k\leq t-4}\gamma_k^2 \ .
\end{gathered}
\label{eq:mu_recurrence_rel}
\end{equation}
\end{enumerate}
\end{theorem}
This will basically follow from the computation of 
$\det \N(z)$ from Section~\ref{sect:determinant} 
together with Cramer's rule.
The hypothesis that $\det \C$ is non-zero will also give us the following.
\begin{corollary}
Let $\H$ be an arbitrary non-zero vector in $\mathbb{C}^\ell$.
Then the equation
\begin{equation}
\N(z)\E(z)\;=\;\H
\end{equation}
has a unique solution.
If $\H=\sum_t h_t \C_t$ then the solution is explicitly given by
\begin{equation}
\E(z)\;=\;\sum_{t} h_t \E_t(z) \ ,
\end{equation}
where $\E_{t}(z)$ denote the solution from the preceding theorem.
\end{corollary}
In what follows, the notation and remarks initially mirror those of Section~\ref{sect:determinant}.  
Throughout this section we fix $t\in\mathbb{Z}_\ell$.
For 
$s\in\mathbb{Z}_\ell$,
define
\begin{equation}
\mathcal{T}_s
\;=\;
\Bigl\{\tau\colon \mathbb{Z}_\ell\setminus \{s\}\to \{0,1,2\}\Bigr\} \ .
\end{equation}
We will also represent $\tau\in\mathcal{T}$ 
by the string over the alphabet 
$\{0,1,2\}$ 
of length $\ell-1$ given by
\begin{equation}\label{eq:tau_string-inhomog}
\tau(s+1)\tau(s+2)\cdots\tau(s-1) \ .
\end{equation}
(Note that the initial and terminal index are different from Section~\ref{sect:determinant}.)
\begin{remark}
As in Section~\ref{sect:determinant}, 
this notation is introduced for the following reason.
From expression~\eqref{eq:O_cols} 
together with multilinearity of the determinant, we find that
\begin{align}\label{eq:Cramer+det_multilinear}
&\det
\left[\begin{array}{ccccccc}
\N_1&\cdots&\N_{s-1}&\C_t&\N_{s+1}&\cdots&\N_\ell
\end{array}\right]\\
&\;=\;
\sum_{\tau\in\mathcal{T}_s}
\det
\left[\begin{array}{ccccccc}
\C_{1;\tau(1)}&\cdots&\C_{s-1;\tau(s-1)}&\C_t&\C_{s+1;\tau(s+1)}&\cdots&\C_{\ell;\tau(\ell)}
\end{array}\right] \ . \notag
\end{align}
For each 
$\tau\in\mathcal{T}_s$, 
multilinearity of the determinant 
and equality~\eqref{eq:C_mp_cols} give
\begin{align}\label{eq:det_monomial-inhomog}
&\det\left[\begin{array}{ccccccc}
\C_{1;\tau(1)}&\cdots&\C_{s-1;\tau(s-1)}&\C_t&\C_{s+1;\tau(s+1)}&\cdots&\C_{\ell;\tau(\ell)}
\end{array}\right]\\
&\;=\;
c(\tau)
\det\left[\begin{array}{ccccccc}
\C_{1+\tau(1)}&\cdots&\C_{s-1+\tau(s-1)}&\C_t&\C_{s+1+\tau(s+1)}&\cdots&\C_{\ell+\tau(\ell)}
\end{array}\right]\notag
\end{align}
where
\begin{equation}
c(\tau)
\;=\;
\prod_{1\leq k\leq \ell : k\neq s} c_k^{\tau(k)} \ .
\end{equation}
Thus again, we will 
(i) ascertain for which $\tau$ the determinant on the right-hand side of~\eqref{eq:det_monomial-inhomog} is non-zero,
and 
(ii) for these $\tau$ we will compute $c(\tau)$.
\end{remark}
For 
$\tau\in\mathcal{T}_s$ 
define 
$\upsilon_\tau\colon\mathbb{Z}_\ell\to\mathbb{Z}_\ell$ 
by
\begin{equation}
\upsilon_\tau(k)\;=\;k+\tau(k) \ .
\end{equation}
This gives a one-to-one correspondence between $\mathcal{T}_s$ and the set
\begin{equation}
\mathcal{U}_{s}
\;=\;
\Bigl\{\upsilon\colon \mathbb{Z}_\ell\setminus\{s\}\to\mathbb{Z}_\ell \ \big\vert \
k\leq \upsilon(k)\leq k+2 \mod \ell, \ \forall k\in\mathbb{Z}_\ell\setminus\{s\} \Bigr\} \ .
\end{equation}
Given $\upsilon\in\mathcal{U}_s$, denote by $\tau_\upsilon$ the element of $\mathcal{T}_s$ corresponding to $\upsilon$. 
As in Section~\ref{sect:determinant}, we say that $\tau\in\mathcal{T}_s$ is {\it decreasing} 
at $k\in\mathbb{Z}_\ell$ if property~\eqref{eq:tau_decreasing_at_k} is 
satisfied. 
(Observe that, for certain $k$, only one of the two statements in~\eqref{eq:tau_decreasing_at_k} may be defined.)
Otherwise we say that $\tau$ is {\it non-decreasing}.
Let
\begin{equation}\label{eq:tau_non-decr+omit_t}
\mathcal{T}_{s,t}^+
\;=\;
\left\{\tau\in\mathcal{T}_s \ : 
\begin{array}{c}
\ \tau \ \mbox{non-decreasing and, whenever they are defined,} \\  
\tau(t-2)-2, \tau(t-1)-1, \tau(t)\neq 0
\end{array}
\right\}
\end{equation}
and denote by $\mathcal{U}_{s,t}^+$ the corresponding subset of $\mathcal{U}_s$.
Then it is straightforward to show that
\begin{equation}
\mathcal{U}_{s,t}^+
\;=\;
\Big\{\upsilon\in\mathcal{U}_s \ : \ \upsilon \ \mbox{injective, and} \ \upsilon(k)\neq t, \forall k\in\mathbb{Z}_\ell\Bigr\} \ .
\end{equation}
Given $\upsilon\in \mathcal{U}_{s}$ we may extend it to $\mathbb{Z}_\ell$ via
\begin{equation}
\bar\upsilon(k)
\;=\;
\left\{\begin{array}{ll}
\upsilon(k) & k\in\mathbb{Z}_\ell\setminus\{s\}\\
t & k=s
\end{array}\right. \ .
\end{equation}
Thus, $\upsilon\in\mathcal{U}_{s,t}^+$ implies that
$\upsilon$ is a bijection between the sets 
$\mathbb{Z}_\ell\setminus\{s\}$ and 
$\mathbb{Z}_\ell\setminus\{t\}$,
and we consequently get the following.
\begin{proposition}
For $\upsilon\in\mathcal{U}_s$, 
$\bar\upsilon\in\mathcal{S}_\ell$ 
only if $\upsilon\in\mathcal{U}_{s,t}^+$.
\end{proposition}
Observe that the function $\bar\upsilon$ denotes the lower index in the right-hand side of equation~\eqref{eq:det_monomial-inhomog}.
The preceding proposition automatically implies the following. 
\begin{corollary}\label{cor:decr+non-decr-inhomog}
Let $\tau\in \mathcal{T}_s$. 
If $\tau\notin\mathcal{T}_{s,t}^+$ then
\begin{equation}
\det
\left[\begin{array}{ccccccc}
\C_{1;\tau(1)}&\cdots&\C_{s-1;\tau(s-1)}&\C_t&\C_{s+1;\tau(s+1)}&\cdots&\C_{\ell;\tau(\ell)}
\end{array}\right]
\;=\;0 \ .
\end{equation}
Otherwise $\tau\in\mathcal{T}_{s,t}^+$ and
\begin{align}
&\det
\left[\begin{array}{ccccccc}
\C_{1;\tau(1)}&\cdots&\C_{s-1;\tau(s-1)}&\C_t&\C_{s+1;\tau(s+1)}&\cdots&\C_{\ell;\tau(\ell)}
\end{array}\right]\notag\\
&\;=\;
(-1)^{\sgn(\bar\upsilon)}
c(\tau)
\det\C
 \ ,
\end{align}
where 
$\sgn(\bar\upsilon)$ denotes parity of the number of adjacent transpositions of $\bar\upsilon$.
\end{corollary}
This resolves the first problem (i). 
Now we wish to compute $c(\tau)$ for $\tau\in\mathcal{T}_{s,t}^+$.
The first step is to observe, by the conditions given 
in~\eqref{eq:tau_non-decr+omit_t}, that we have the following.
(See Figure~\ref{fig:tau-omitted_value}.)
\begin{proposition}\label{prop:tau-omitted_value}
Let $\tau\in\mathcal{T}_{s,t}^+$.
Then
\begin{enumerate}
\item
For $s=t$,
\begin{equation}\label{eq:case-s=t}
\tau(t-2)\tau(t-1)\;=\;
00
\ \ \mbox{or} \ \  
^1_0 2 \ \ .
\end{equation}
\item
For $s=t-1$,
\begin{equation}
\tau(t-2)\tau(t)\;=\;
^1_0 \phantom{\!}^2_1 \ \ .
\end{equation}
\item
For $s=t-2$,
\begin{equation}
\tau(t-1)\tau(t)\;=\;
22
\ \ \mbox{or} \ \
0 ^2_1 \ \ .
\end{equation}
\item
For 
$s\neq t,t-1$ or $t-2$, 
either
\begin{equation}
\tau(t-2)\tau(t-1)\tau(t)\;=\; 
00^2_1 \ \ ,
\end{equation}
or
\begin{equation}
\tau(t-2)\tau(t-1)\tau(t)\;=\;
^1_0\!22 \ \ .
\end{equation}
\end{enumerate}
\end{proposition}
%
\begin{figure}[h] 
\centering
\psfrag{t}[][]{${\scriptstyle t}$}
\psfrag{t-1}[][]{${\scriptstyle t-1}$} 
\psfrag{t-2}[][]{${\scriptstyle t-2}$} 
\psfrag{0}[][]{${\scriptstyle 0}$} 
\psfrag{1}[][]{${\scriptstyle 1}$} 
\psfrag{2}[][]{${\scriptstyle 2}$} 
\begin{subfigure}[t]{1.0\textwidth}
\centering
\includegraphics[width=0.725\textwidth]{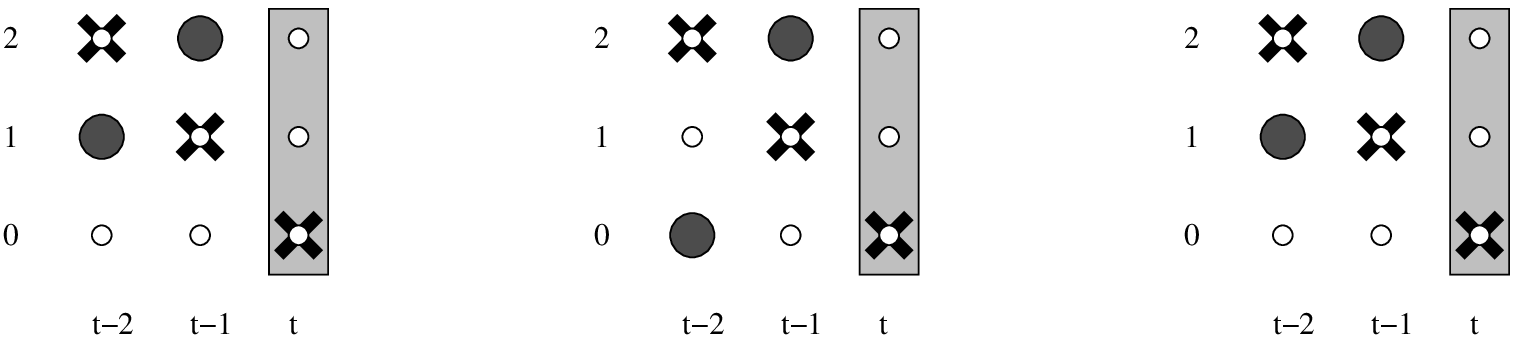}
\caption{Permitted configurations for $s=t$}
\label{fig:s=t}
\end{subfigure}
\\[10pt]
\begin{subfigure}[t]{1.0\textwidth}
\centering
\includegraphics[width=1.0\textwidth]{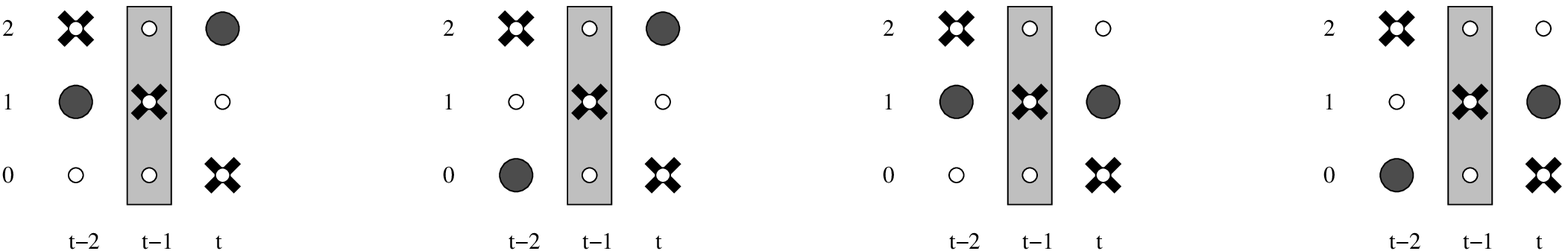}
\caption{Permitted configurations for $s=t-1$}
\label{fig:s=t-1}
\end{subfigure}
\\[10pt]
\begin{subfigure}[t]{1.0\textwidth}
\centering
\includegraphics[width=0.725\textwidth]{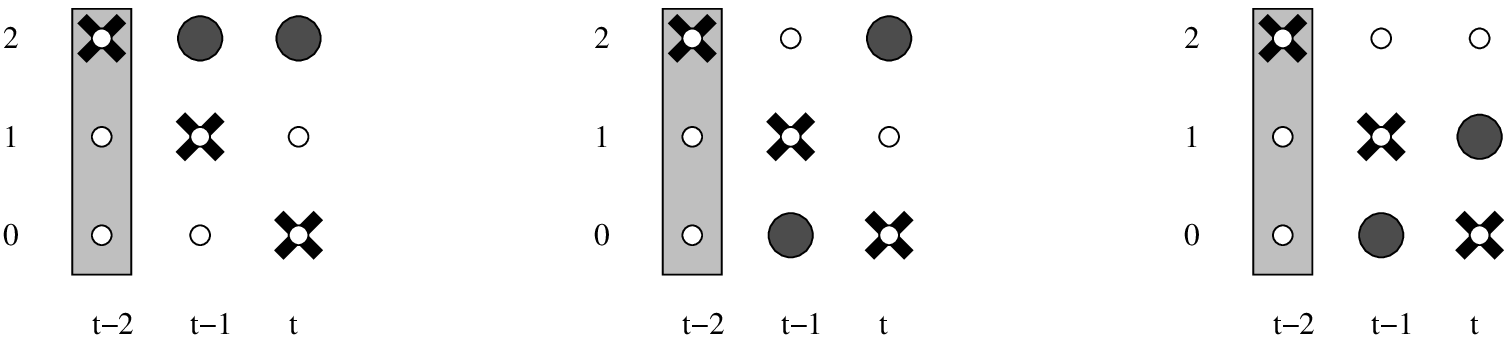}
\caption{Permitted configurations for $s=t-2$}
\label{fig:s=t-2}
\end{subfigure}
\\[10pt]
\begin{subfigure}[t]{1.0\textwidth}
\includegraphics[width=1.0\textwidth]{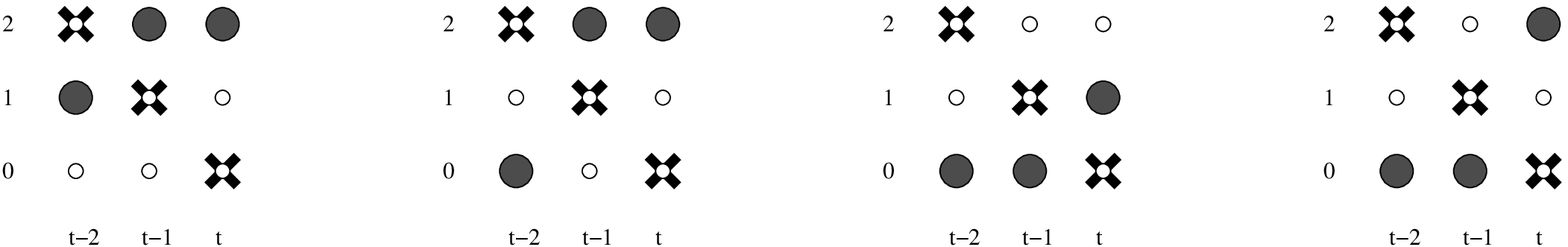}
\caption{Permitted configurations for $s\neq t, t-1, t-2$.}
\label{fig:s_general}
\end{subfigure}
\caption[slope]{\label{fig:config_permitted}
The graph of configurations $\tau$ in $\mathcal{T}_{s,t}^+$ restricted to the interval $[t-2,t]$ 
in the cases $s=t$, $s=t-1$, $s=t-2$ and otherwise, as prescribed by Proposition~\ref{prop:tau-omitted_value}.
The grey rectangles are to remind the reader that $\tau$ has no value there, 
the black crosses denote values which are prohibited by equation~\eqref{eq:tau_non-decr+omit_t},
and the black circles denote the value of $\tau$.
}
\label{fig:tau-omitted_value}
\end{figure} 
For 
$\tau\in\mathcal{T}_{s,t}^+$ 
we require $\tau$ to be non-decreasing on the interval $[s+1,s-1]$ in $\mathbb{Z}_\ell$.
Therefore, applying Lemma~\ref{lem:key} together with Proposition~\ref{prop:tau-omitted_value}, 
we also get the following.
\begin{corollary}\label{cor:tau-omitted_value}
Let $\tau\in\mathcal{T}_{s,t}^+$.
\begin{enumerate}
\item
For $s=t$, the string corresponding to $\tau$ has the form
\begin{equation}
0^{\ell-1} \ \ \mbox{or} \ \ 1^{\lambda_1}b\cdots 1^{\lambda_{r-1}}b1^{\lambda_r}2 \ .
\end{equation}
\item
For $s=t-1$, the string corresponding to $\tau$ has the form
\begin{equation}
1^{\lambda_1}b\cdots 1^{\lambda_{r-1}}b1^{\lambda_r} \ .
\end{equation}
\item
For $s=t-2$, the string corresponding to $\tau$ has the form
\begin{equation}
2^{\ell-1} \ \ \mbox{or} \ \ 01^{\lambda_1}b\cdots 1^{\lambda_{r-1}}b1^{\lambda_r} \ .
\end{equation}
\item
For $s\neq t,t-1,t-2$, the string corresponding to $\tau$ either has the form
\begin{equation}
0^\kappa1^{\lambda_1}b\cdots 1^{\lambda_{r-1}}b1^{\lambda_r}
\end{equation}
or
\begin{equation}
1^{\lambda_1}b\cdots 1^{\lambda_{r-1}}b1^{\lambda_r}2^\mu \ .
\end{equation}
where 
\begin{equation}
\kappa\;=\;\card [s+1,t-1] \ \ \mbox{and} \ \ \mu\;=\;\card [t-1,s-1] \ .
\end{equation}
\end{enumerate}
\end{corollary}
We say that 
$\tau\in\mathcal{T}_{s,t}^+$ 
is of 
{\it type $(\kappa,\mu)$}
if the string corresponding to $\tau$ 
has the form
\begin{align}\label{eq:def-type}
0^\kappa1^{\lambda_1}b\cdots 1^{\lambda_{r-1}}b1^{\lambda_r}2^{\mu} \ ,
\end{align}
where 
$\kappa$, $\lambda_1,\ldots,\lambda_r$, and $\mu$, are non-negative integers.
(See figure~\ref{fig:configuration_eg}.)
Observe that, since $0^{\ell-1}$ is the unique type $(\ell-1,0)$ and $2^{\ell-1}$ is the unique type $(0,\ell-1)$,
Corollary~\ref{cor:tau-omitted_value} can be rephrased as follows: 
For $\tau\in\mathcal{T}_{s,t}^+$,
\begin{enumerate}
\item 
if $s=t$, then $\tau$ is either of type $(\ell-1,0)$ or type $(0,1)$;
\item 
if $s=t-1$, then $\tau$ of type $(0,0)$;
\item 
if $s=t-2$, then $\tau$ is either of type $(0,\ell-1)$ and type $(1,0)$;
\item 
if $s\neq t,t-1,t-2$, then $\tau$ is either of type $(\kappa,0)$ or type $(0,\mu)$, where $\kappa=\card[s+1,t-1]$ and $\mu=\card[t-1,s-1]$.
\end{enumerate}
If we define $\kappa(s,t)$ and $\mu(s,t)$ as in~\eqref{eq:kappa_def} and~\eqref{eq:mu_def}
then this becomes the following.
\begin{corollary}\label{cor:tau-omitted_value-types}
$\mathcal{T}_{s,t}^+$ 
consists exactly of types $(\kappa,0)$ and types $(0,\mu)$.
\end{corollary}
\begin{proposition}\label{prop:b_and_beta}
Suppose that 
\begin{equation}
c^p_s(z)\;=\;\gamma^p_sz^p \ \mbox{for some} \ \gamma^p_s\in\mathbb{C} \ , \qquad \forall s=1,2,\ldots,\ell \ , \ \forall p=0,1,2 \ .
\end{equation}
If we define
\begin{align}
b_{s,t}^{\kappa,0}(z)
&\;=\;
\sum_{\substack{\tau\in\mathcal{T}_{s,t}^+ \\ \tau \ \mbox{\tiny type} \ (\kappa,0)}}
(-1)^{\sgn(\bar\upsilon_\tau)}c(\tau)(z)
\end{align}
and
\begin{align}
b_{s,t}^{0,\mu}(z)
&\;=\;
\sum_{\substack{\tau\in\mathcal{T}_{s,t}^+ \\ \tau \ \mbox{\tiny type} \ (0,\mu)}}
(-1)^{\sgn(\bar\upsilon_\tau)}c(\tau)(z)
\end{align}
then
\begin{equation}\label{eq:b_in_terms_of_beta}
b_{s,t}^{\kappa,0}(z)
\;=\;z^{(\ell-1)-\kappa}u_{s,t}^{\kappa,0} \quad \mbox{and} \quad 
b_{s,t}^{0,\mu}(z)
\;=\;z^{(\ell-1)+\mu}u_{s,t}^{0,\mu} \ ,
\end{equation}
where 
$u_{s,t}^{\kappa,0}$ 
and 
$u_{s,t}^{0,\mu}$ 
satisfy respectively the recurrence relations
~\eqref{eq:kappa_recurrence_rel} and
~\eqref{eq:mu_recurrence_rel}.
\begin{comment}
i.e.,
\begin{equation}
\begin{gathered}
u_{s,t}^{\kappa,0}
=
\frac{-1}{\gamma_{s}^{0}}\left(
\gamma_{s-1}^{1}u_{s-1, t}^{\kappa+1, 0}+\gamma_{s-2}^{2}u_{s-2, t}^{\kappa+2, 0}
\right)\\
u_{t,t}^{\ell-1,0}=\prod_{t+1\leq k\leq t-1}\gamma_k^0\qquad\qquad
u_{t+1,t}^{\ell-2,0}=-\gamma_{t}^1\prod_{t+2\leq k\leq t-1}\gamma_k^0
\end{gathered}
\label{eq:sum_(kappa,0)-recurrence}
\end{equation}
and
\begin{equation}
\begin{gathered}
u_{s, t}^{0,\mu}
=
\frac{-1}{\gamma_{s}^{2}}\left(
\gamma_{s+1}^{1}u_{s+1, t}^{0, \mu+1}+\gamma_{s+2}^{0}u_{s+2, t}^{0, \mu+2}
\right)\\
u_{t-2,t}^{0,\ell-1}=\prod_{t-1\leq k\leq t-3}\gamma_k^2\qquad \qquad
u_{t-3,t}^{0,\ell-2}=-\gamma_{t-2}^{1}\prod_{t-1\leq k\leq t-4}\gamma_k^2
\end{gathered}
\label{eq:sum_(0,mu)-recurrence}
\end{equation}
\end{comment}
\end{proposition}
%
\begin{figure}[ht] 
\begin{center} 
\psfrag{1}[][]{$1$} 
\psfrag{2}[][]{$2$} 
\psfrag{3}[][]{$3$} 
\psfrag{4}[][]{$4$} 
\psfrag{5}[][]{$5$} 
\psfrag{6}[][]{$6$} 
\psfrag{7}[][]{$7$} 
\psfrag{8}[][]{$8$} 
\psfrag{9}[][]{$9$} 
\psfrag{10}[][]{$10$} 
\includegraphics[width=4.1in]{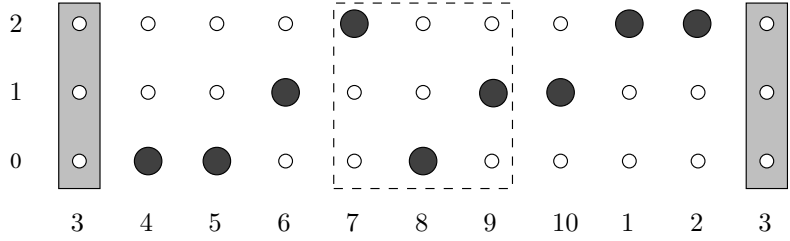}
\end{center} 
\caption[slope]{\label{fig:configuration_eg}
The graph of the non-decreasing configuration $\tau$ 
of type $(2,2)$ in $\mathcal{T}^+_{3,7}$ corresponding 
to the string $0^2 1b1^2 2^2$. 
The grey rectangles are used to remind the reader that 
$\tau$ has no value at $s=3$, while the dotted rectangle 
is to draw attention to the values of $\tau$ 
at $t=7$ : namely that $\tau(t+2)-2,\tau(t+1)-1,\tau(t)$ are non-zero.} 
\end{figure} 
\begin{remark}
Observe that the function
\begin{equation}
b_{t-1,t}^{0,0}(z)\;=\;
\sum_{\substack{\tau\in\mathcal{T}_{t-1,t}^+ \\ \tau \ \mbox{\tiny type} \ (0,0)}}
(-1)^{\sgn(\bar\upsilon_\tau)}c(\tau)
\end{equation}
is used in both recurrence relations
~\eqref{eq:kappa_recurrence_rel} and
~\eqref{eq:mu_recurrence_rel}.
Thus, it is implicitly taken in the statement of the above proposition that 
$b_{t-1,t}^{0,0}$ has the form
$z^{\ell-1}u_{t-1,t}^{0,0}$
and $u_{t-1,t}^{0,0}$ satisfies both recurrences.
\end{remark}
\begin{proof}
First, let us show equation~\eqref{eq:b_in_terms_of_beta}.
Observe that
\begin{equation}\label{eq:c_in_terms_of_gamma}
c(\tau)
\;=\;
\gamma(\tau)z^{\sum_k\tau(k)} \qquad
\mbox{where} \qquad
\gamma(\tau)
\;=\;
\prod_{1\leq k\leq\ell: k\neq s}\gamma_k^{\tau(k)} \ .
\end{equation}
By equation~\eqref{eq:def-type}, if $\tau$ is of type $(\kappa,\mu)$ then
\begin{equation}\label{eq:sum_of_tau}
\sum_{1\leq k\leq \ell : k\neq s}\tau(k)
\;=\;
\ell-1-\kappa+\mu \ .
\end{equation}
Thus,
$\tau$ being of type $(\kappa,0)$ implies that $\deg_z c(\tau)=\ell-\kappa-1$; while
$\tau$ being of type $(0,\mu)$ implies that $\deg_z c(\tau)=\ell-1+\mu$.
Hence equation~\eqref{eq:b_in_terms_of_beta} is satisfied.
We now prove the recurrence relations~\eqref{eq:kappa_recurrence_rel} and~\eqref{eq:mu_recurrence_rel} by induction.
First let us show~\eqref{eq:kappa_recurrence_rel}.
We start by considering the cases $s=t,t+1$.
For $s=t$, there is a unique 
$\tau\in\mathcal{T}_{s,t}^+$.
Namely,
\begin{equation}
\tau(t+1)\tau(t+2)\cdots\tau(t-2)\tau(t-1)\;=\;0^{\ell-1} \ .
\end{equation}
Therefore
\begin{equation}
\gamma(\tau)
\;=\;
\gamma_{t+1}^{\tau(t+1)}\cdots\gamma_{t-2}^{\tau(t-2)}\gamma_{t-1}^{\tau(t-1)}
\;=\;
\prod_{t+1\leq k\leq t-1} \gamma_{k}^0 \ .
\end{equation}
Also observe that 
$\bar\upsilon_\tau=\mathrm{id}$. 
Therefore $\sgn(\bar\upsilon_\tau)=0$.
Hence, 
we find that
\begin{equation}
\sum_{\substack{\tau\in\mathcal{T}_{t,t}^+\\ \tau \ \mbox{\tiny type} \  (\ell-1,0)}}
(-1)^{\sgn(\bar{\upsilon}_\tau)}\gamma(\tau)
\;=\;
(-1)^{\sgn(\bar{\upsilon}_\tau)}\gamma(\tau)
\;=\;
\prod_{t+1\leq k\leq t-1} \gamma_{k}^0 \ .
\end{equation}
Next, for $s=t+1$, there is a unique 
$\tau\in \mathcal{T}_{s,t}^+$ of type $(\ell-2,0)$. 
Namely, 
\begin{equation}
\tau(t+2)\tau(t+3)\cdots\tau(t-1)\tau(t)\;=\;0^{\ell-2}1 \ .
\end{equation}
From which we see that
\begin{equation}
\gamma(\tau)
\;=\;
\gamma_{t+2}^{\tau(t+2)}\gamma_{t+3}^{\tau(t+3)}\cdots \gamma_{t-1}^{\tau(t-1)}\gamma_{t}^{\tau(t)}
\;=\;
\left(\prod_{t+2\leq k\leq t-1}\gamma_{k}^0 \right)
\gamma_{t}^{1} \ .
\end{equation}
Also observe that 
$\bar\upsilon_\tau$ 
is a single transposition given in cycle notation by $(t\ t+1)$.
Therefore $\sgn(\bar\upsilon_\tau)=1$.
Thus, 
we find that
\begin{equation}
\sum_{\substack{\tau\in\mathcal{T}_{t+1,t}^+\\ \tau \ \mbox{\tiny type} \  (\ell-2,0)}}
(-1)^{\sgn(\bar{\upsilon}_\tau)}\gamma(\tau)
\;=\;
(-1)^{\sgn(\bar{\upsilon}_\tau)}\gamma(\tau)
\;=\;
-
\left(\prod_{t+2\leq k\leq t-1}\gamma_{k}^0\right) 
\gamma_{t}^{1}
\end{equation}
and the statement holds for $s=t+1$.
Now assume that the statement holds for $t\leq r< s$.
We wish to show the same holds for $r=s$.
Observe that we can decompose $\mathcal{T}_{s,t}^+(\kappa(s,t),0)$ further into
those $\tau$ whose central string terminates with $1$, 
and those whose central string terminates with $b$,
{\it i.e.}, let 
\begin{align}
\mathcal{T}_{s,t}^+\left(\kappa(s,t),0\right)^1
&=
\left\{
0^{\kappa}\varsigma1 : 
\varsigma(t)\cdots\varsigma(s-2)=1^{\lambda_1}b1^{\lambda_2}b\cdots 1^{\lambda_r}
\ \mbox{some} \ \lambda_1,\ldots
\right\}\\
\mathcal{T}_{s,t}^+\left(\kappa(s,t),0\right)^0
&=
\left\{
0^{\kappa}\varsigma b : 
\varsigma(t)\cdots\varsigma(s-3)=1^{\lambda_1}b1^{\lambda_2}b\cdots 1^{\lambda_r}
\ \mbox{some} \ \lambda_1,\ldots
\right\}
\end{align}
Also observe that
\begin{itemize}
\item
$\mathcal{T}_{s,t}^+\left(\kappa(s,t),0\right)^1$ 
is in one-to-one correspondence with 
$\mathcal{T}_{s-1,t}^+\left(\kappa(s,t)+1,0\right)$;
\item
$\mathcal{T}_{s,t}^+\left(\kappa(s,t),0\right)^0$ 
is in one-to-one correspondence with 
$\mathcal{T}_{s-2,t}^+\left(\kappa(s,t)+2,0\right)$.
\end{itemize}
Moreover, 
if 
$\tau\in\mathcal{T}_{s,t}^+\left(\kappa(s,t),0\right)^1$ 
corresponds to 
$\tau'\in\mathcal{T}_{s-1,t}^+\left(\kappa(s,t)+1,0\right)$,
then
\begin{align}
\gamma_{s}^{0}\cdot
\gamma(\tau)
&\;=\;
\gamma_{s}^{0}\cdot
\gamma_{s+1}^{\tau(s+1)}\cdots
\gamma_{t-1}^{\tau(t-1)}\cdot\gamma_{t}^{\tau(t)}\cdots 
\gamma_{s-1}^{\tau(s-1)}\\
&\;=\;
\gamma_{s}^{0}
\left(\prod_{s+1\leq k\leq t-1}\gamma_k^0\right)
\left(\prod_{t\leq k\leq s-2}\gamma_{k}^{\varsigma(k)}\right)
\gamma_{s-1}^{1}\\
&\;=\;
\left(\prod_{s\leq k\leq t-1}\gamma_k^0\right)
\left(\prod_{t\leq k\leq s-2}\gamma_{k}^{\varsigma(k)}\right)
\gamma_{s-1}^{1}
\ \;=\; \
\gamma_{s-1}^{1}\cdot\gamma(\tau') \ .
\end{align}
Also, the corresponding permutations 
$\bar{\upsilon}_\tau$ 
and 
$\bar{\upsilon}_{\tau'}$ 
differ by a single transposition interchanging $t$ and $s$.
(In cycle notation: 
$\bar{\upsilon}_{\tau}=(t \ s)\bar{\upsilon}_{\tau'}$.)
Since $(t\ s)$ can be expressed as the product of 
$2\card[t,s]-3$ 
adjacent transpositions 
$(t\ t+1)(t+1\ t+2)\cdots(s-2\ s-1)(s-1\  s)(s-2\ s-1)\cdots (t\ t+1)$ 
it follows that 
$\sgn(\bar\upsilon_\tau)$ 
and 
$\sgn(\bar\upsilon_{\tau'})$ have different parities.

Similarly,
if 
$\tau\in\mathcal{T}_{s,t}^+\left(\kappa(s,t),0\right)^0$ 
corresponds to 
$\tau'\in\mathcal{T}_{s-2,t}^+\left(\kappa(s,t)+2,0\right)$,
then
\begin{align}
\gamma_{s}^{0}\cdot
\gamma(\tau)
&\;=\;
\gamma_{s}^{0}\cdot
\gamma_{s+1}^{\tau(s+1)}\cdots
\gamma_{t-1}^{\tau(t-1)}\cdot\gamma_{t}^{\tau(t)}\cdots 
\gamma_{s-1}^{\tau(s-1)}\\
&\;=\;
\gamma_{s}^{0}
\left(\prod_{s+1\leq k\leq t-1}\gamma_k^0\right)
\left(\prod_{t\leq k\leq s-3}\gamma_{k}^{\varsigma(k)}\right)
\gamma_{s-2}^{2} \cdot \gamma_{s-1}^{0}\\
&\;=\;
\left(\prod_{s-1\leq k\leq t-1}\gamma_k^0\right)
\left(\prod_{t\leq k\leq s-3}\gamma_{k}^{\varsigma(k)}\right)
\gamma_{s-2}^{2}
\ \;=\; \
\gamma_{s-2}^{2}\cdot
\gamma(\tau')
\end{align}
and the corresponding permutations 
$\bar{\upsilon}_\tau$ and $\bar{\upsilon}_{\tau'}$ 
differ by a single transposition interchanging $t$ and $s$.
(In cycle notation: $\bar{\upsilon}_{\tau}=(t \ s)\bar{\upsilon}_{\tau'}$.)
Thus 
$\sgn(\bar\upsilon_\tau)$ 
and 
$\sgn(\bar\upsilon_{\tau'})$ 
again have different parities.

Consequently, setting $\kappa=\kappa(s,t)$ and observing that $s=t+\ell-\kappa-1$,
\begin{align}
\gamma_{s}^{0}\cdot u_{s,t}^{\kappa,0}
&\;=\;
\sum_{\tau\in\mathcal{T}_{s,t}^+\left(\kappa,0\right)}
(-1)^{\sgn(\bar{\upsilon}_\tau)}\gamma_{s}^{0}\cdot\gamma(\tau)\\
&\;=\;
\sum_{\tau\in\mathcal{T}_{s,t}^+\left(\kappa,0\right)^1}
(-1)^{\sgn(\bar{\upsilon}_\tau)}\gamma_{s}^{0}\cdot\gamma(\tau)
+\sum_{\tau\in\mathcal{T}_{s,t}^+\left(\kappa,0\right)^0}
(-1)^{\sgn(\bar{\upsilon}_\tau)}\gamma_{s}^{0}\gamma(\tau)\\
&\;=\;
\sum_{\tau'\in\mathcal{T}_{s-1,t}^+\left(\kappa+1,0\right)}
(-1)^{\sgn(\bar{\upsilon}_{\tau'})+1}\gamma_{s-1}^{1}\cdot\gamma(\tau')\notag\\
&\qquad \qquad +\sum_{\tau\in\mathcal{T}_{s-2,t}^+\left(\kappa+2,0\right)}
(-1)^{\sgn(\bar{\upsilon}_{\tau'})+1}\gamma_{s-2}^{2}\cdot\gamma(\tau')\\
&\;=\;
-\gamma_{s-1}^{1}u_{s-1,t}^{\kappa+1,0}
-\gamma_{s-2}^{2}u_{s-2,t}^{\kappa+2,0} \ .
\end{align}
Now we consider the second recurrence relation~\eqref{eq:mu_recurrence_rel}
For $s=t-2$ there exists a unique $\tau\in\mathcal{T}_{s,t}^+$.
Namely
\begin{equation}
\tau(t-1)\tau(t)\cdots\tau(t-4)\tau(t-3)
\;=\;
2^{\ell-1} \ .
\end{equation}
Therefore
\begin{equation}
\gamma(\tau)
\;=\;
\gamma_{t-1}^{\tau(t-1)}\gamma_{t}^{\tau(t)}\cdots\gamma_{t-4}^{\tau(t-4)}\gamma_{t-3}^{\tau(t-3)}
\;=\;
\prod_{t-1\leq k\leq t-3} \gamma_{k}^2 \ .
\end{equation}
Now observe that, since 
$\bar\upsilon_\tau$ 
can be written as $\nu^2$ 
where 
$\nu(k)=k+1\mod\ell$, 
it follows that 
$\sgn(\bar\upsilon_\tau)=2\sgn(\nu)=2\ell-2$.
Hence
\begin{equation}
\sum_{\substack{\tau\in\mathcal{T}_{t-2,t}^+\\ \tau \ \mbox{\tiny type} \  (0,\ell-1)}}
(-1)^{\sgn(\bar{\upsilon}_\tau)}\gamma(\tau)
\;=\;
(-1)^{\sgn(\bar{\upsilon}_\tau)}\gamma(\tau)
\;=\;
\prod_{t-1\leq k\leq t-3}\gamma_k^2 \ .
\end{equation}
Now consider the case 
$s=t-3$.
Then there is a unique 
$\tau\in\mathcal{T}_{s,t}^+$ of type $(0,\ell-2)$.
Namely
\begin{equation}
\tau(t-2)\tau(t-1)\cdots\tau(t-5)\tau(t-4)
\;=\;
12^{\ell-2} \ .
\end{equation} 
From which, we find that
\begin{equation}
\gamma(\tau)
\;=\;
\gamma_{t-2}^{\tau(t-2)}\gamma_{t-1}^{\tau(t-1)}\cdots\gamma_{t-5}^{\tau(t-5)}\gamma_{t-4}^{\tau(t-4)}
\;=\;
\gamma_{t-2}^1\left(\prod_{t-1\leq k\leq t-4}\gamma_k^2\right) \ .
\end{equation}
Next, observe that 
$\bar\upsilon_\tau$ 
can be written, in cycle notation, as 
$( t \ t-1 )\nu^2$ 
and thus 
$\sgn(\bar\upsilon_\tau)=2\ell-1$.
It follows that
\begin{equation}
\sum_{\substack{\tau\in\mathcal{T}_{t-3,t}^+\\ \tau \ \mbox{\tiny type} \  (0,\ell-2)}}
(-1)^{\sgn(\bar{\upsilon}_\tau)}\gamma(\tau)
\;=\;
(-1)^{\sgn(\bar{\upsilon}_\tau)}\gamma(\tau)
\;=\;
-
\gamma_{t-2}^1\left(\prod_{t-1\leq k\leq t-4}\gamma_k^2\right) \ .
\end{equation}
Now assume that the statement holds for $s+1\leq r< t$.
We decompose $\mathcal{T}_{s,t}^+\left(0,\mu\right)$ into
\begin{align}
\mathcal{T}_{s,1}^+\left(0,\mu\right)^1
&\;=\;
\left\{1\varsigma 2^\mu : 
\varsigma(s+2)\cdots\varsigma(t-2)=1^{\lambda_1}b1^{\lambda_r}\cdots 1^{\lambda_r}
\ \mbox{some} \ \lambda_1,\ldots
\right\}\\
\mathcal{T}_{s,t}^+\left(0,\mu\right)^0
&\;=\;
\left\{b\varsigma 2^\mu : 
\varsigma(s+3)\cdots\varsigma(t-2)=1^{\lambda_1}b1^{\lambda_r}\cdots 1^{\lambda_r}
\ \mbox{some} \ \lambda_1,\ldots
\right\}
\end{align}
As before, we observe that
\begin{itemize}
\item
$\mathcal{T}_{s,t}^+\left(0,\mu(s,t)\right)^1$ 
is in one-to-one correspondence with 
$\mathcal{T}_{s+1,t}^+\left(0,\mu(s,t)+1\right)$;
\item
$\mathcal{T}_{s,t}^+\left(0,\mu(s,t)\right)^0$ is in one-to-one correspondence with $\mathcal{T}_{s+2,t}^+\left(0,\mu(s,t)+2\right)$.
\end{itemize}
Moreover, 
if 
$\tau\in\mathcal{T}_{s,t}^+\left(0,\mu(s,t)\right)^1$ 
corresponds to 
$\tau'\in\mathcal{T}_{s+1,t}^+\left(0,\mu(s,t)+1\right)$
then
\begin{align}
\gamma_{s}^{2}\cdot
\gamma(\tau)
&\;=\;
\gamma_{s}^{2}\cdot
\gamma_{s+1}^{\tau(s+1)}\cdots\gamma_{t-1}^{\tau(t-1)}\cdot\gamma_{t}^{\tau(t)}\cdots \gamma_{s-1}^{\tau(s-1)}\\
&\;=\;
\gamma_{s}^{2}\cdot
\gamma_{s+1}^1
\left(\prod_{s+2\leq k\leq t-2}\gamma_k^{\varsigma(k)}\right)
\left(\prod_{t-1\leq k\leq s-1}\gamma_k^{2}\right)\\
&\;=\;
\gamma_{s+1}^{1}
\left(\prod_{s+2\leq k\leq t-2}\gamma_k^{\varsigma(k)}\right)
\left(\prod_{t-1\leq k\leq s}\gamma_k^{2}\right)
\ \;=\; \
\gamma_{s+1}^{1}\cdot
\gamma(\tau')
\end{align}
and the corresponding permutations $\bar\upsilon_\tau$ and $\bar\upsilon_{\tau'}$ differ by a single transposition interchanging $s$ and $s+1$. 
(In cycle notation: $\bar\upsilon_\tau=\bar\upsilon_{\tau'} ( \ s \ s+1 \ )$.)
Thus $\sgn(\bar\upsilon_\tau)=\sgn(\bar\upsilon_{\tau'})+1$.

Similarly, 
if $\tau\in\mathcal{T}_{s,t}^+\left(0,\mu(s,t)\right)^0$ corresponds to $\tau'\in\mathcal{T}_{s+2,t}^+\left(0,\mu(s,t)+2\right)$
then
\begin{align}
\gamma_{s}^{2}\cdot
\gamma(\tau)
&\;=\;
\gamma_{s}^{2}\cdot
\gamma_{s+1}^{\tau(s+1)}\cdots\gamma_{t-1}^{\tau(t-1)}\cdot\gamma_{t}^{\tau(t)}\cdots \gamma_{s-1}^{\tau(s-1)}\\
&\;=\;
\gamma_{s}^{2}\cdot
\gamma_{s+1}^2\cdot 
\gamma_{s+2}^{0}
\left(\prod_{s+3\leq k\leq t-2}\gamma_{k}^{\varsigma(k)}\right)
\left(\prod_{t-1\leq k\leq s-1}\gamma_{k}^{2}\right)\\
&\;=\;
\gamma_{s+2}^{0}
\left(\prod_{s+3\leq k\leq t-2}\gamma_{k}^{\varsigma(k)}\right)
\left(\prod_{t-1\leq k\leq s+1}\gamma_{k}^{2}\right)
\ \;=\; \
\gamma_{s+2}^{0}\cdot\gamma(\tau')
\end{align}
and the corresponding permutations 
$\bar\upsilon_\tau$ and $\bar\upsilon_{\tau'}$ 
differ by a single transposition interchanging $s$ and $s+2$.
(In cycle notation: $\bar\upsilon_\tau=\bar\upsilon_{\tau'}( \ s \ s+2 \ )$.)
This transposition can be expressed as the product of three adjacent transpositions.
Thus $\sgn(\bar\upsilon_\tau)=\sgn(\bar\upsilon_{\tau'})+1$.
Consequently, setting $\mu=\mu(s,t)$ and observing that $s=t-\ell+\mu-1$,
\begin{align}
\gamma_{s}^{2}\cdot 
u_{s,t}^{0,\mu}
&\;=\;
\sum_{\tau\in\mathcal{T}_{s,t}^+\left(0,\mu\right)}(-1)^{\sgn(\bar{\upsilon}_\tau)}\gamma_{s}^{2}\cdot\gamma(\tau)\\
&\;=\;
\sum_{\tau\in\mathcal{T}_{s,t}^+\left(0,\mu\right)^1}(-1)^{\sgn(\bar{\upsilon}_\tau)}\gamma_{s}^{2}\cdot\gamma(\tau)
+\sum_{\tau\in\mathcal{T}_{s,t}^+\left(0,\mu\right)^0}(-1)^{\sgn(\bar{\upsilon}_\tau)}\gamma_{s}^{2}\cdot\gamma(\tau)\\
&\;=\;
\sum_{\tau'\in\mathcal{T}_{s+1,t}^+\left(0,\mu+1\right)}(-1)^{\sgn(\bar{\upsilon}_{\tau'})+1}\gamma_{s+1}^{1}\cdot\gamma(\tau')\notag\\
&\qquad\qquad+\sum_{\tau'\in\mathcal{T}_{s+2,t}^+\left(0,\mu+2\right)}(-1)^{\sgn(\bar{\upsilon}_{\tau'})+1}\gamma_{s+2}^{0}\cdot\gamma(\tau')\\
&\;=\;
-\gamma_{s+1}^1 u_{s+1,t}^{0,\mu+1}
-\gamma_{s+2}^{0}u_{s+2,t}^{0,\mu+2} \ .
\end{align}
\end{proof}
Using the preceding analysis we can relate the coefficients 
$u_{s,t}^{\kappa,0}$ and $u_{s,t}^{0,\mu}$ 
to the coefficients of the quadratic polynomial $v$ from Theorem~\ref{thm:det-general}. 
\begin{proposition}\label{prop:alpha_in_terms_of_beta}
Let $v$ denote the quadratic polynomial 
from Theorem~\ref{thm:det-general}.
Then, for any $s\in\mathbb{Z}_\ell$, the degree 1 coefficient $v_1$ of $v$ satisfies the following relation:
\begin{equation}\label{eq:alpha_in_terms_of_beta}
v_1\;=\;\gamma_s^0u_{s,s}^{0,1}+\gamma_s^1 u_{s,s+1}^{0,0}+\gamma_s^2 u_{s,s+2}^{1,0} \ .
\end{equation}
\end{proposition}
\begin{proof}
First let us comment on notation. 
Fix $s\in\mathbb{Z}_\ell$.
Below we will denote 
elements of $\mathcal{T}^+$ by $\bar\tau$, 
for $t=s,s+1,s+2$, 
and elements of $\mathcal{T}^+_{s,t}$
will be denoted by $\tau$.
Observe that
\begin{equation}
\left\{\bar\tau\in \mathcal{T}^+ : \tau\not\equiv 0,2\right\}
\;=\;
\biguplus_{p=0,1,2}
\left\{\bar\tau\in\mathcal{T}^+ : \tau\not\equiv 0,2,\tau(s)=p\right\} \ .
\end{equation}
\begin{claim}
We have the following:

\vspace{5pt}

$\left\{\bar\tau\in\mathcal{T}^+ : \bar\tau\not\equiv 0,2,\bar\tau(s)=0\right\}$
is in one-to-one correspondence with
$\mathcal{T}_{s,s}^+\left(0,1\right)$;

\vspace{5pt}

$\left\{\bar\tau\in\mathcal{T}^+ : \bar\tau\not\equiv 0,2,\bar\tau(s)=1\right\}$ 
is in one-to-one correspondence with
$\mathcal{T}^+_{s,s+1}\left(0,0\right)$;

\vspace{5pt}

$\left\{\bar\tau\in\mathcal{T}^+ : \bar\tau\not\equiv 0,2,\bar\tau(s)=2 \right\}$
is in one-to-one correspondence with
$\mathcal{T}_{s,s+2}^+\left(1,0\right)$.
\end{claim}

\vspace{5pt}

\noindent
{\it Proof of Claim:}
We prove the middle correspondence.
The left-hand side can be identified with the set of strings satisfying the properties 
$\tau(s-1)\neq 2$ and $\tau(s+1)\neq 0$, 
{\it i.e.}, those strings of the form 
\begin{equation}
\tau(s+1)\cdots\tau(s-1)
\;=\;
1^{\lambda_1}b1^{\lambda_2}b\cdots 1^{\lambda_{r-1}}b1^{\lambda_r} \ .
\end{equation}
But strings of this form correspond to elements of $\mathcal{T}^+_{s,s+1}\left(0,0\right)$.
The other correspondences follow similarly.
/\!/

\vspace{5pt}

Under each of these correspondences we also trivially have that 
$\upsilon_{\bar\tau}=\bar\upsilon_\tau$.
In particular 
$\sgn(\upsilon_{\bar\tau})=\sgn(\bar\upsilon_\tau)$.
We also have, for $p=0,1$ or $2$, that if $\tau\in\mathcal{T}^+$ satisfies $\tau(s)=p$ then 
\begin{equation}
\gamma(\bar\tau)\;=\;\gamma_s^p \gamma(\tau) \ .
\end{equation}
Therefore, from the expression from Theorem~\ref{thm:det-general},
\begin{align}
v_1
&\;=\;
\sum_{\bar\tau\in\mathcal{T}^+:\bar\tau\not\equiv 0,2}(-1)^{\sgn(\upsilon_{\bar\tau})}\gamma(\bar\tau)\\
&\;=\;
\sum_{\tau\in\mathcal{T}^+_{s,s}\left(0,1\right)}
(-1)^{\sgn(\bar\upsilon_\tau)}\gamma_s^{0}\gamma(\tau)
+\sum_{\tau\in\mathcal{T}^+_{s,s+1}\left(0,0\right)}
(-1)^{\sgn(\bar\upsilon_\tau)}\gamma_s^{1}\gamma(\tau)\notag\\
&\qquad\qquad\qquad+\sum_{\tau\in\mathcal{T}^+_{s,s+2}\left(1,0\right)}
(-1)^{\sgn(\bar\upsilon_\tau)}\gamma_s^{2}\gamma(\tau)\\
&\;=\;
\gamma_s^0u_{s,s}^{0,1}+\gamma_s^1 u_{s,s+1}^{0,0}+\gamma_s^2 u_{s,s+2}^{1,0} \ .
\end{align}
\end{proof}
We can also now prove the main result in this section.
\begin{proof}[Proof of Theorem~\ref{thm:inhomog_eq-sol}]
Let
\begin{equation}
u_s(z)\;=\;
\sum_{\tau\in\mathcal{T}_{s,t}^+}
(-1)^{\sgn(\bar\upsilon_\tau)}
c(\tau) \ .
\end{equation}
Then by equations~\eqref{eq:Cramer+det_multilinear} 
and~\eqref{eq:det_monomial-inhomog}, we find that
\begin{align}
&\det
\left[\begin{array}{ccccccc}
\N_1&\cdots&\N_{s-1}&\C_t&\N_{s+1}&\cdots&\N_\ell
\end{array}\right]\\
&\;=\;
\sum_{\tau\in\mathcal{T}_s}
\det
\left[\begin{array}{ccccccc}
\C_{1;\tau(1)}&\cdots&\C_{s-1;\tau(s-1)}&\C_t&\C_{s+1;\tau(s+1)}&\cdots&\C_{\ell;\tau(\ell)}
\end{array}\right]\notag\\
&\;=\;
\sum_{\tau\in\mathcal{T}_{s,t}^+}
(-1)^{\sgn(\bar\upsilon_\tau)}
c(\tau)
\det
\left[\begin{array}{cccc}
\C_{1}&\C_{2}&\cdots&\C_{\ell}
\end{array}\right]\\
&\;=\;
\det\C \times u_s(z) \ . \label{eq:Cramer_numerator}
\end{align}
Theorem~\ref{thm:det-general} implies that
\begin{align}
\det \N(z)\;=\;\det\C\times v(z^\ell) \ . \label{eq:Cramer_denominator}
\end{align}
Applying Cramer's rule to equation~\eqref{eq:inhomog}, 
together with the equalities~\eqref{eq:Cramer_numerator} 
and~\eqref{eq:Cramer_denominator},
a solution $\E(z)$ exists provided that $\det \C$ is non-zero, 
and in this case the $s$th entry is given by
\begin{align}
E_s(z)
&\;=\;
\frac{u_s(z)}{v(z^\ell)} \ .
\end{align}
By Proposition~\ref{prop:b_and_beta},
\begin{equation}
u_s(z)\;=\;
\left\{
\begin{array}{ll}
z^{(\ell-1)-\kappa}u_{s,t}^{\kappa,0}+z^{(\ell-1)+\mu}u_{s,t}^{0,\mu}
& s\neq t-1\\
z^{\ell-1}u_{t-1,t}^{0,0}
& s=t-1
\end{array}
\right.
\end{equation}
where $u_{s,t}^{\kappa,0}$ and $u_{s,t}^{0,\mu}$ satisfy 
properties~\eqref{eq:kappa_recurrence_rel} and~\eqref{eq:mu_recurrence_rel}. 
\end{proof}

\end{document}